\documentclass[12pt]{amsart}
\usepackage{amssymb, amsthm, amsmath, amstext}
\usepackage{mathrsfs}  
\usepackage{bm}        
\usepackage{mathtools} 
\usepackage[russian,english]{babel}

\usepackage[left=1.28in,top=.8in,right=1.28in,bottom=.8in]{geometry}
\usepackage{comment}
\usepackage{graphicx}
\usepackage[all]{xy}
\usepackage{xcolor}
\usepackage[utf8]{inputenc}

\theoremstyle{plain}
\newtheorem{theorem}{Theorem}[section]
\newtheorem{prop}[theorem]{Proposition}
\newtheorem{lemma}[theorem]{Lemma}

\theoremstyle{definition}

\theoremstyle{remark}

\newcommand{\sheaf}[1]{\mathscr{#1}}

\newcommand{\PP}{\sheaf{P}}
\newcommand{\XX}{\sheaf{X}}

\newcommand{\DD}{\sheaf{D}}
\newcommand{\QQ}{\sheaf{Q}}

\newcommand{\UU}{\sheaf{U}}

\newcommand{\Z}{\mathbb Z}

\usepackage{hyperref}
\usepackage{color}

\hypersetup{colorlinks=true,linkcolor=blue,anchorcolor=blue,citecolor=blue,letterpaper=true}

\author{Zitong Pei}
\address{Department of Mathematics, Virginia Tech, Blacksburg, VA 24061, USA}
\email{zitongpei@vt.edu}

\begin{document}
\pagestyle{plain}
\title{REDUCED UNITARY WHITEHEAD GROUPS OVER FUNCTION FIELDS OF $p$-adic CURVES}

\date{}

\begin{abstract}     Let $F_0$ be the function field of a curve over a $p$-adic field $K,$ and let $F$ be a quadratic extension over $F_0$. Let $A$ be a central simple algebra over $F$ of period $2,$ and let $\tau$ be a $F/F_0$-involution on $A$. 
 We show the triviality of the reduced unitary Whitehead group $SK_1U( A, \tau)$ if $p\neq 2$. 
\end{abstract} 
 
\maketitle

\def\ZZ{${\mathbf Z}$}
\def\ih{${\mathbf H}$}
\def\RR{${\mathbf R}$}
\def\IF{${\mathbf F}$}
\def\QQ{${\mathbf Q}$} 
\def\IP{${\mathbf P}$}

\section{Introduction}
Let $F_0$ be an infinite field.  Let $G$ be a smooth connected linear algebraic group over $F_0$  and  $G(F_0)$ denote the group of 
$F_0$-rational points of $G.$ Two points $x,y \in G(F_0)$ are defined to be $R$-equivalent if there is a rational map $f: \mathbb{A}_{F_0}^1 
\dashrightarrow G$ such that $f(0)=x$ and $f(1)=y.$ The definition of $R$-equivalence was introduced by Y. I. Manin in \cite{manin1986} 
when studying cubic hypersurfaces. The $R$-equivalence is actually an equivalent relation, denoted by $'\sim'$, on $G(F_0)$
(cf. \cite{1977reduced}). Let $RG(F_0)$ denote the equivalence  class of the identity $e \in G(F_0)$. Then $RG(F_0)$ 
is a normal subgroup of $G(F_0)$ and there is a bijection of sets between $G(F_0)/\sim$ and $G(F_0)/RG(F_0)$
(cf. \cite{1977reduced}, \cite{CT1977}).

Suppose that $G$ is semi-simple, simply connected, isotropic, and simple over $F_0$. Let $G(F_0)^{+}$ denote the normal subgroup of $G(F_0)$ generated by the $F_0$-rational points of the unipotent radicals of parabolic $F_0$-subgroups of $G$. The quotient group ${G(F_0)}/{G(F_0)^+}$ is called the Whitehead group for $G$ over $F_0$ and denoted by $Wh(G, F_0)$. It is known that $Wh(G, F_0) \cong G(F_0)/RG(F_0)$(cf. \cite{AST2009}). The Kneser-Tits conjecture (\cite{Tits1978}) predicted the triviality of the Whitehead group $Wh(G, F_0)$. 

In \cite{platonov1975}, V. P. Platonov constructed a counterexample to the Kneser-Tits conjecture, thereby demonstrating that the conjecture has a negative answer. Platonov's counterexample is based on the algebraic principle of $Wh(G, F_0)$, which inspires the exploration of the underlying algebraic background introduced in the following discussion.

Let $A$  be  a central simple $F_0$-algebra. Let $SL_{1}(A)$ denote the subgroup $\{a \in A^*|Nrd_{A/F_0}(a)=1\}$ and $[A^*, A^*]$
 denote the commutator subgroup of $A^*$. The quotient group ${SL_{1}(A)}/{[A^*, A^*]},$ 
 denoted by $SK_1(A),$ is called the {\it reduced Whitehead } group of $A$ over $F$. 
 Furthermore, $SK_1(A)$ only depends on the Brauer class of $A$ in $Br(F_0)$ (cf. \cite{salgebra}).
  In \cite{1977reduced}, it was proved that $Wh(G, F_0) \cong SK_1(A)$ if $G = SL_1(A).$

Let  $F$ be   a quadratic field extension of $F_0$ and $A$ is a central simple $F$-algebra. 
 Assume that $A$ has a unitary involution $\tau$ such that $ F_0=\{x \in F|\tau(x)=x\}.$
  In this case, we also say that $A$ has a $F/F_0$-involution $\tau$. 
  Let $\Sigma_{\tau}'(A^*)=\{a\in A^*|Nrd_{A/F}(a) \in F_0\}$ and $\Sigma_{\tau}(A^*)$ the subgroup of $A^*$  generated by 
  $ \{b\in A^*|\tau(b)=b\} $. 
   Then $\Sigma_{\tau}'(A^*)$ is a subgroup of $A^*$, and $\Sigma_{\tau}(A^*)$ is a normal subgroup of $\Sigma'_{\tau}(A^*)$. 
Let $D$ be a central division $F$-algebra such that $A \cong M_n(D)$ for some integer $n$. Then there is a unitary involution $\tau'$ on $D$ such that $\tau'|_{F}=\tau|_{F}$ and  $\Sigma'_{\tau}(A)/\Sigma_{\tau}(A) \cong \Sigma'_{\tau'}(D)/\Sigma_{\tau'}(D)$ (\cite[Lemma 2]{1974simple}). Let $V$ and $\Phi$ be, respectively, a right finite dimensional $D$-vector space and a skew hermitian form with respect to $\tau'$ on $V \times V$ of positive Witt index. Let $U(\Phi)$ denote the unitary group of $\Phi$, $TU(\Phi)$ denote the subgroup of those elements of $U(\Phi)$ generated by the unitary transvections, and $SU(\Phi)$ denote the subgroup of those elements of $U(\Phi)$ whose reduced norms are equal to unity. If the dimension $dim_{D}(V) >1,$ then $\Sigma_{\tau}'(A)/\Sigma_{\tau}(A) \cong \Sigma_{\tau'}'(D)/\Sigma_{\tau'}(D) \cong SU(\Phi)/TU(\Phi)$ (cf. \cite[II, Chapter 4, \S3.2]{salgebra}). Thus, we can always view $A$ as a central division $F$-algebra. The quotient group $\Sigma_{\tau}'(A)/\Sigma_{\tau}(A),$ denoted by $SK_1U(A, \tau)$, is called the {\it 
reduced unitary Whitehead group} of $A$ over $F$. In \cite{1977reduced}, it was proved that $W(G, F_0)\cong SK_1U( A, \tau)$ if $G =SU(\Phi)$.

It is known that if the index of $A$ is a prime, then both groups 
$SK_1(A)$ and $SK_1U(A, \tau)$ are trivial (\cite{Wang}, \cite{1974simple}). 
There are also examples of fields $F$ and algebras $A$ such that 
$SK_1(A)$ and $SK_1U(A, \tau)$ are non trivial.

If the cohomological dimension of $F$ is at most 2, then both the groups are trivial (\cite{1975commutator}, \cite{1974simple}). 
A conjecture of Suslin asserts that if the cohomological dimension of $F$ is  3, 
then both the groups $SK_1(A)$ and $SK_1U(A, \tau)$ are trivial (\cite{Suslin}).  

If $F$ is a complete discretely valued field with residue field a field of cohomological dimension 2, then 
the cohomological dimension of $F$ is 3 and in this case the above conjecture of Suslin is answered 
in affirmative(\cite{soman}). 

The triviality of $SK_1(A)$ or $Sk_1U(A, \tau)$ depends on the choice of the field $F$. In this paper, we primarily focus on the function field of a $p$-adic curve. In \cite{bh}, Nivedita established the triviality of $SK_1(A)$ for the function fields of $p$-adic curves under certain conditions. 

\begin{theorem}(\cite[Theorem 13.8]{bh})
Let $F$ be the function field of a $p$-adic curve. Let $D$ be a central division $F$-algebra of prime exponent $\ell$. If $\ell \neq 2$ , $\ell \neq p,$ and $F$ contains a primitive $\ell^{2th}$ root of unity, then $SK_1(A)$ is trivial.

\end{theorem}

The main aim of this paper is to prove the following theorem .

\begin{theorem}(cf. Theorem \ref{thm9.2} )
     Let $F_0$ be the function field of a $p$-adic curve. Let $F$ be a quadratic field extension of $F_0$. Let $A$ be a central simple $F$-algebra of the exponent $2$. If $A$ has a $F/F_0$-involution $\tau$ and $p \neq 2$, then $SK_1U(\tau, A)$ is trivial. 
  
\end{theorem}

Here is an outline of the structure of the paper. For any $a \in \Sigma_{\tau}'(A^*)$, it needs to prove that $a \in \Sigma_{\tau}(A^*)$. We will use patching techniques developed by Harbater, Hartmann and Krashen (cf. \cite{h2009}, \cite{h2010}, \cite{h2014}, \cite{h2015}) to find three extensions of the function field $F$ satisfying certain conditions (cf. Theorem \ref{thm9.1} ). Let $\XX$ be a regular model of the $p$-adic curve, and $\XX_o$ be the reduced scheme of the closed fiber of $\XX$.  The scenario of patching is to locally study cases at generic points and closed points of $\XX_o$, then deduces the case of the field $F$ globally. In fact, we can find a decomposition of $a$ such that $a$ is a product of $\tau$-symmetric elements by globally constructing three extensions of $F$. In Section \ref{section2}, we first find a suitable model $\XX$ of the $p$-adic curve by resolution of singularities, which leads to conditions satisfying the application of patching techniques. In Section \ref{section3}, we prove weak approximations over global fields. In Section \ref{section4}, we provide an alternative proof of the triviality of  $SK_1U(A, \tau)$ for complete discretely valued fields, which is a local version of the main theorem. This is for the local data at generic points of $\XX_o$.  In Section \ref{section5}, we study the division algebras over two dimensional complete fields. It is used to deal with nodal points of $\XX_o$. In Section \ref{section6} \& \ref{section7} we give explicit constructions of  the local data both at nodal points, non-nodal closed points, and generic points of $\XX_o$. In Section \ref{section8}, we finish all the local data needed in patching techniques through finding suitable non-empty open subsets of $\XX_o$. In Section \ref{section9}, we use patching techniques to deduce the main conclusion.

\section{Preliminaries} \label{section2}
Let  $p > 3$ be  a prime and  $K$ be a  $p$-adic field.
Let $F_0$ be the function field of a curve over $K$ and 
$F = F_0(\sqrt{d})$ a quadratic field extension.  
Let $D$ be a central division algebra over $F$ with a $F/F_0$-involution. 

Let $T$ be the  valuation ring of $K$ and $k$ the residue field of $K$.
Let $\XX_0$ be regular proper model of $F_0$ over $R$ and 
$\XX$ the integral closure of $\XX_0$ in $F$.

We say that  $D$ is  {\it ramified} at  codimension one point  $\eta$ of $\XX_0$,
if $D$ is ramified at  one of the codimension points of $\XX$ lying over $\eta$.
The {\it ramification locus } ram$_{\XX_0}(D)$ of $D$ on $\XX_0$ is defined as 
the set of codimension one points of $\XX_0$ where $D$ is ramified.

There exists a regular proper model $\XX_0$ of $F_0$ such that 
the union of ram$_{\XX_0}(D)$,  the supp$(d)$ and the closed fibre $X_0$
 is a union of regular curves with normal crossings (cf.  \cite{lipman1975},  \cite[Theorem 9.2.16, Proposition 10.1.8]{2002al}).
Further  the integral closure $\XX$  of $\XX_0$ in $F$ is a regular proper model of $F$.
Let $P \in X_0$ be a closed point.
 Let $m_P$ be the maximal ideal at $P$ on $\XX_0$
  and $\kappa(P)$ be the residue field at $P$.

Suppose $P$ is a regular point. 
 Since $X_0$ is a union of regular curves, there exists a unique codimension one point 
 $\eta$ of $X_0$ such that $P \in \overline{\eta}$.

\section{Global fields}\label{section3}
In this section we prove certain weak approximations over global fields which are used in the proof of our main theorem. 
We begin with the following well known result. 

\begin{lemma}
\label{biquad}
 Let $K$  be a field with char$(K) \neq 2$.
Let $u, d, w \in k^*$.  If $u \in N_{K(\sqrt{d}, \sqrt{w})/K(\sqrt{d})}(K(\sqrt{d}, \sqrt{w})^*)$, then 
$u \in N_{ K(\sqrt{w})/K}(K(\sqrt{w})^*) N_{K(\sqrt{wd})/K}(K(  \sqrt{wd})^*)$.
\end{lemma}

\begin{proof} Let $q = <1,  -w> - u<1, -wd>$.  
Then $q \otimes K(\sqrt{d}) \simeq <1 , -w> < 1, -u>$.  Since  $u$ is norm from the extension $K(\sqrt{d}, \sqrt{w})/K(\sqrt{d})$, 
$q$ is isotropic over $K(\sqrt{d})$. Since the discriminant of $q$ is $d$, $q$ is isotropic over $K$ (cf. \cite[Chapter VII, Corollary 3.3]{lam2005}).
Hence there exists $x_1, x_2, x_3, x_4 \in K$ (not all zero) such that $x_1^2 - wx_2^2 = u(x_3^2 - wdx_4^2)$.
In particular $u \in N_{ K(\sqrt{w})/K}(K(\sqrt{w})^*) N_{K(\sqrt{wd})/K}(K(  \sqrt{wd})^*)$.
\end{proof}

%

\begin{lemma}
\label{weakapp} Let $\kappa$ be a global field of characteristic not 2 and $d, w \in \kappa^*$.
Let $S$ be a finite set of paces of $\kappa$. 
Let $u \in N_{\kappa(\sqrt{d})/\kappa}(\kappa(\sqrt{d})^*)N_{\kappa(\sqrt{w})/\kappa}(\kappa(\sqrt{w})^*$.
For each $\nu \in S$, let $y_\nu \in \kappa_\nu(\sqrt{d})$ and $z_\nu \in \kappa_\nu(\sqrt{w})$
be such that $N_{\kappa_\nu(\sqrt{d})/\kappa_\nu}(y_\nu)N_{\kappa_\nu(\sqrt{w})/\kappa_\nu}( z_\nu) $ is 
close to $ u$. 
Then there exist $y \in \kappa(\sqrt{d})$ and $z \in \kappa(\sqrt{w})$
be such that  $y$ is close to $y_\nu$ and $z$ is close to $z_\nu$ for all $\nu \in S$ and 
$N_{\kappa(\sqrt{d})/\kappa}(y)N_{\kappa(\sqrt{w})/\kappa}( z) = u$.
\end{lemma}
\begin{proof}
Since $N_{\kappa(\sqrt{w})/\kappa}( z)^{-1} = N_{\kappa(\sqrt{w})/\kappa}( z^{-1})$, 
the equation $N_{\kappa(\sqrt{d})/\kappa}(y)N_{\kappa(\sqrt{w})/\kappa}( z) = u$ is same as 
$N_{\kappa(\sqrt{d})/\kappa}(y) = u N_{\kappa(\sqrt{w})/\kappa}( z^{-1})$.

Let $Q$ be the  projective quadric   given by the quadratic form $y_1^2 - d y_2^2 - u z_1^2 + uw z_2^2 $.
Since $u \in N_{\kappa(\sqrt{d})/\kappa}(\kappa(\sqrt{d})^*)N_{\kappa(\sqrt{w})/\kappa}(\kappa(\sqrt{w})^*)$,
$Q(\kappa) \neq \emptyset$.  Hence $Q$  satisfies weak approximation (cf. \cite[2.1.4]{Har2004}).
Thus  there  exist  elements $y \in \kappa(\sqrt{d})$ and $z \in \kappa(\sqrt{w})$ satisfying the required
properties. 
\end{proof}

 \begin{lemma}
 \label{globalfield1} Let $\kappa$ be a global field of characteristic not 2.
 Let $  d_0, w_0  \in \kappa^*$ and  $S_0$ a finite set of places of $\kappa$.  
   For each place $\nu \in S_0$,  suppose we have given 
  $x_{\nu} \in \kappa_{\nu}^*$,  
  $y_{\nu} \in \kappa_{1\nu} =    \kappa_{\nu}(\sqrt{x_\nu})$
  such that  
  $y_{\nu} \in   N_{\kappa_{1\nu}(   \sqrt{w_0})/\kappa_{1\nu}}(\kappa_{1\nu}
  (  \sqrt{w_0})^* ) N_{\kappa_{1\nu}(  \sqrt{d_0w_0})/\kappa_{1\nu} }
  ( \kappa_{1\nu}(  \sqrt{d_0w_0})^*).$
  Then there exist 
  $x \in \kappa$  and $y \in \kappa_1 = \kappa(\sqrt{x})$ 
  such that

  i)  $ x$ is close to $x_{\nu} $ and $y $ is close to $y_\nu$  for all $\nu \in S$
  
  ii)  $y \in   N_{\kappa_1(  \sqrt{w_0})/\kappa_1 }(\kappa_1(  
  \sqrt{w_0})^* ) N_{\kappa_1(  \sqrt{d_0w_0})/\kappa_1}
  ( \kappa_1(  \sqrt{d_0w_0})^*)$.
 \end{lemma}
 
 \begin{proof}  Let $x\in \kappa$ be close to $x_\nu$ for all $\nu \in S$ and $\kappa_1 = \kappa(\sqrt{x})$.
 Then $\kappa_1 \otimes \kappa_\nu  = \kappa_{1\nu}$.
 Let  $z_{1\nu} \in \kappa_{1\nu}(   \sqrt{w_0})^*$ and 
 $z_{2\nu} \in \kappa_{1\nu}(   \sqrt{d_0w_0})^*$ such that 
 $$y_\nu = N_{\kappa_{1\nu}(   \sqrt{w_0} )/\kappa_{1\nu}}(z_{1\nu})
 N_{\kappa_{1\nu}(  \sqrt{d_0w_0} )/\kappa_{1\nu}}(z_{1\nu}).$$ 
 
 Let $z_1  \in \kappa_1( \sqrt{w_0})^*$ and 
 $z_{2} \in \kappa_1 ( \sqrt{d_0w_0} )^*$ close to $z_{1\nu}$ and 
 $z_{2\nu}$ respectively for all $\nu \in S$. Let 
 $$ y = N_{\kappa_1(  \sqrt{w_0})/\kappa_1}(z_1)
 N_{\kappa_1(  \sqrt{d_0w_0} )/\kappa_1}(z_2).$$ 
 Then $x$ and $y$ have the required properties.  
  \end{proof}

  \begin{lemma}
 \label{globalfield} Let $\kappa$ be a global field of characteristic not 2.
 Let $u_0,  w_0, d_0  \in \kappa^*$ and  $S_0$ a finite set of places of $\kappa$. 
Suppose that $u_0 \in N_{\kappa(\sqrt{d_0}, \sqrt{w_0})/\kappa( \sqrt{d_0}) }
(\kappa(  \sqrt{d_0}, \sqrt{w_0})^*)$.
  For each place $\nu \in S_0$,  suppose we have given 
  $x_{\nu} \in \kappa_{\nu}^*$,  
  $y_{1\nu} \in \kappa_{1\nu} =   \kappa_{\nu}(\sqrt{x_{\nu}})$,
  $y_{2\nu} \in \kappa_{2\nu} = \kappa_\nu(\sqrt{w_0})$ and $y_{3\nu} \in \kappa_{3} = 
  \kappa_\nu(\sqrt{d_0w_0})$
  such that  
  
  i) $ \prod_i N_{\kappa_{i\nu}/\kappa_\nu}(y_{i\nu}) = u_0$ 
  
  ii) $y_{1\nu} \in 
 N_{\kappa_{1\nu}(   \sqrt{w_0})/\kappa_{1\nu}}(\kappa_{1\nu}(   \sqrt{w_0})^* ) 
 N_{\kappa_{1\nu}(  \sqrt{d_0w_0})/\kappa_{1\nu}}
  ( \kappa_{1\nu}(   \sqrt{d_0w_0})^*)$. 
 
  Then there exist   $x \in \kappa^*$  and $y_1 \in \kappa_1 = \kappa(\sqrt{x})$, 
  $y_2 \in \kappa_2 = \kappa(\sqrt{w_0})$, $y_3 \in \kappa_3 = \kappa(\sqrt{d_0w_0})$ 
  such that 
  
  i)   $ x$ is close to $x_{\nu} $ and  $y_i$  is close to $y_{i\nu}$     for all $\nu \in S$ and $i = 1, 2, 3$.
  
  ii)  $ \prod_i N_{\kappa_{i}/\kappa}(y_i) = u_0$ 
  
  iii)  $y_1  \in N_{\kappa_1(  \sqrt{w_0})/\kappa_1}(\kappa_1
  (   \sqrt{w_0})^* ) N_{\kappa_1(  \sqrt{d_0w_0})/\kappa_1}
  ( \kappa_1(   \sqrt{d_0w_0})^*)$. 
 \end{lemma}
 
 \begin{proof}    By (\ref{globalfield1}), there exist  $x \in \kappa$ and
  $y_1 \in  \kappa_1 = \kappa(\sqrt{x})$
  such that $x $ is close to $x_{\nu}$, $y_1$ is close to $y_{1\nu}$  for all $\nu \in S$
  and $y_1 = N_{\kappa_1(  \sqrt{w_0})/\kappa_1}(z_1)
 N_{\kappa_1( \sqrt{d_0w_0} )/\kappa_1}(z_2) $  for some $
 z_1  \in \kappa_1(  \sqrt{w_0})^*$ and 
 $z_{2} \in \kappa_1 ( \sqrt{d_0w_0})^*$. 
 
 Let  $u_1  = N_{\kappa_1/\kappa}(y_1)$, 
  $u_{12} =  N_{\kappa_1(   \sqrt{w_0})/\kappa}(z_1)$
 and $u_{13} =  N_{\kappa_1(  \sqrt{d_0w_0})/\kappa}(z_2)$.
 Then  
  $u_1 = u_{12} u_{13}$  and $u_{12} \in  N_{\kappa_2/\kappa}(\kappa_2^*)$ and 
  $u_{13} \in N_{\kappa_3/\kappa}(\kappa_3^*)$.   
  Let $u_2 = u_0 u_1^{-1}$.  Since $u_0 \in N_{\kappa(\sqrt{d_0}, \sqrt{w_0})/\kappa( \sqrt{d_0}) }
(\kappa(  \sqrt{d_0}, \sqrt{w_0})^*)$, by (\ref{biquad}), $u_0 \in N_{ \kappa_2/\kappa}(\kappa_2^*) N_{\kappa_3/\kappa}
(\kappa_3^*)$. 

Since $u_1  = u_{12}u_{13} \in  N_{ \kappa_2/\kappa}(\kappa_2^*) N_{\kappa_3/\kappa}
(\kappa_3^*)$, $u_2 = u_0u_1^{-1} \in N_{ \kappa_2/\kappa}(\kappa_2^*) N_{\kappa_3/\kappa}
(\kappa_3^*)$. 
  Since  $y_1$ is close to $y_{1\nu}$ for all $\nu \in S$, 
  $N_{\kappa_{2\nu}/\kappa_\nu}(y_{2\nu}) N_{\kappa_{3\nu}/\kappa_\nu}(y_{3\nu})$ is close to 
  $u_2$. Hence,  by (\ref{weakapp}), there exist $y_2 \in \kappa_2$ and $y_3 \in \kappa_3$ which are close to 
 $y_{2\nu}$ and $y_{3\nu}$ respectively  for all $\nu \in S$ such that 
 $N_{\kappa_2/\kappa}(y_2) N_{\kappa_3/\kappa}(y_3) = u_2$. 
  Then $x_1$, $y_1$, $y_2$ and $y_3$ have 
  the required properties.  
 \end{proof}

\section{Complete discretely valued fields}\label{section4}
Let $R_0$ be a complete discretely valued ring with residue field $\kappa$ a positive characteristic 
global field of characteristic not equal to 2 and  $F_0$   the field of fractions. 
Let $d \in R_0$ be a non-square and $F = F_0(\sqrt{d})$. 

Let $R$ be the integral closure of
$R_0$ in $F$.  Let $D$ be a central division algebra over $F$ with a $F/F_0$-involution $\tau$.

We know that $SK_1U(D, \tau)$ is trivial (\cite[Corollary 4.16, Corollary 4.17]{Y1979}). 
The aim of this section is to show that given $\lambda \in F_0^*$ which is a reduced norm from 
$D$,   there exist $a_1, a_2, a_3 \in F_0^*$ and 
$\mu_i \in L_i = F_0[X]/(X^2 - a_i)$ with some local conditions over the residue field  
 such that $\prod_i N_{L_i/F_0}(\mu_i) = \lambda$ and $ind(D\otimes_{F_0}  L_i) \leq 2$.
This is required for our main result and also  this gives an alternative proof of the fact that 
$SK_1U(A, \tau)$ is trivial.  

%
%
%

\begin{lemma}
\label{dvr-class-4} 
Suppose that the valuation of $d$ is even.  Let $\pi \in R_0$ be a parameter.   Then
$D = (b, c) \otimes (w, \pi)$ for some units $b, c,  w \in R_0$  
\end{lemma}
\begin{proof}
 Since the valuation of $d \in F_0$ is even, the extension $F/F_0$ is unramified.
Let $\pi$ be a parameter of $F_0$.  Then $\pi$ is also a parameter in $F$.
Hence $D = D' + (w, \pi)$ for some $D'$ unramified at $R$ and $w \in R$ a unit (cf. \cite[Lemma 4.1]{PPS}). Since $D$ has a unitary $F/F_0$-involution, $cores_{F/F_0}(D)=0.$ 

Furthermore, $cores_{F/F_0}(w,\pi)=(w, N_{F/F_0}(\pi))=0.$  Hence, $D'$ and $(w, \pi)$ have unitary $F/F_0$-involutions (cf. \cite[Chapter 1, 3.B.]{knus}, \cite[Chapter 30, Appendix, F5*]{algebraII}).

Since $D'$ is unramified and residue field of $F$  is a global field, ind$(D') \leq 2$.
Hence there exists a   algebra $D_0$  over $F_0$ such that $D_0 \otimes F \simeq D'$ (cf. \cite[Proposition 2.22]{knus}).
Since $D'$ is unramified, we can choose $D_0$ to be unramified. Hence
$D_0  = (b,c )$ for some   units $b, c \in R_0$. Since $(w,\pi)$ is also a quaternion algebra,
there exists a quaternion algebra $D_1$ over $F_0$ such that $D_1 \otimes F \simeq (w, \pi)$.
Thus without loss of generality, we assume that $w \in R_0^*$ a unit.  
Hence $D = (b,c) \otimes (w,\pi)$ with $b, c, w \in R_0^*$. 
\end{proof}

\begin{lemma}
\label{dvr-class-2}
 Suppose that the valuation of $d$ is odd.  Then  $D = (b, c)$  for some units $b, c \in R_0$.
\end{lemma}
\begin{proof}
 Since the valuation of $d$ is odd, $F/F_0$ is a ramified extension. 
 Hence, the residue field $\kappa$ does not change  and $ D = D_0 \otimes F$ for some unramified  central simple algebra $D_0$ over $F_0$ (cf. \cite[Lemma 4.1]{PPS}).
 Since the residue field of $F_0$ is a global field, $D_0 = (b,c)$ for some $b, c \in R_0^*$.
 Hence $D = (b, c)$ with $b,c \in R_0^*$. 
\end{proof}

\begin{prop}  
\label{dvr-unit} Suppose  $d $  is a unit in $R_0$ and   $D = (b, c) \otimes (w, \pi)$ for some  units $b, c, w \in R_0$,
$\pi \in R_0$ a parameter. 
Let $u \in R_0$ be a unit.  Suppose that $\bar{u}  \in N_{\kappa(\sqrt{\bar{d}}, \sqrt{\bar{w}})/\kappa( \sqrt{\bar{d}}) }
(\kappa(  \sqrt{\bar{d}}, \sqrt{\bar{w}})^*)$.
Let $S_0$ be a finite set of places of $\kappa$ containing all the places $\nu$ where at least one of 
$\bar{u}$, $\bar{b}, \bar{c}, \bar{w}, \bar{d} $ is not a unit.  
  For each place $\nu \in S_0$,  suppose we have given 
  $x_{\nu} \in \kappa_{\nu} -\kappa_\nu^{*2}$,  $y_{1\nu} \in \kappa_{1\nu} =   \kappa[X]/(X^2 - x_{1\nu})$,
 $y_{2\nu} \in \kappa _{2\nu} =  \kappa_{\nu}[X]/(X^2 - \bar{w})$ and 
 $y_{3\nu}  \in \kappa_{3\nu} =  \kappa_{\nu}[X]/(X^2 - \bar{w}\bar{d})$ such  that 
 
 i) $\prod_{i=1}^3 N_{\kappa_{i\nu}/\kappa_{\nu}}(y_{i\nu}) = \bar{u}$ 
 
 ii) $y_{1\nu}  \in N_{ \kappa_{1\nu}(  \sqrt{\bar{w}}, \sqrt{\bar{d}})/\kappa_{1\nu}( \sqrt{\bar{d}})}
( \kappa_{1\nu}(  \sqrt{\bar{w}}, \sqrt{\bar{d}})^*)$.  

Then there exist  units $a \in R_0$, $\mu_1 \in R_0[X]/(X^2 - a)$, 
 $\mu_2 \in R_0[X]/(X^2 - w)$, $\mu_3 \in R_0[X]/(X^2 - wd)$ such that 
 
 i) $\bar{a}$ is close to $x_{\nu}$, $\bar{\mu}_i$ is close to $y_{i\nu}$ for all $\nu \in S$
 and $i = 1, 2, 3$ 
 
  ii) $\prod_i N_{L_i/F_0}(\mu_i) =  u$, where $L_1 =   F_0[X]/(X^2 - a)$, 
  $L_2 = F_0[X]/(X^2 - w)$ and 
  
 \hskip 8mm  $L_3 = F_0[X]/(X^2 - dw)$
  
  iii) $(b, c) \otimes L_1$ is split
 
 iv) $\mu_i$ is a reduced norm from $D \otimes_{F_0} L_i$ for $i = 1, 2, 3$.  
  \end{prop}
  
  \begin{proof}     Since  $y_{1\nu} \in N_{ \kappa_{1\nu}(  \sqrt{\bar{w}}, \sqrt{\bar{d}})/\kappa_{1\nu}( \sqrt{\bar{d}})}
( \kappa_{1\nu}(  \sqrt{\bar{w}}, \sqrt{\bar{d}})^*)$,    by (\ref{biquad}), \\
$y_1    \in    N_{ \kappa_{1\nu}(  \sqrt{\bar{w}})/\kappa_{1\nu} } ( \kappa_{1\nu}(  \sqrt{\bar{w}})^*) 
N_{ \kappa_{1\nu}(  \sqrt{\bar{wd}})/\kappa_{1\nu} } ( \kappa_{1\nu}(  \sqrt{\bar{wd}})^*) $. 
  Thus, by (\ref{globalfield}), there exist 
   $x \in \kappa$  and $y_1 \in \kappa_1 = \kappa(\sqrt{x})$, 
  $y_2 \in \kappa_2 = \kappa(\sqrt{\bar{w}})$, $y_3 \in \kappa_3 = \kappa(\sqrt{\bar{d}\bar{w}})$ 
  such that 
  
  i)   $ x$ is close to $x_{\nu} $ and  $y_i$  is close to $y_{i\nu}$ for all $\nu \in S_0$ and $i = 1, 2, 3$.
  
  ii)  $ \prod_i N_{\kappa_{i}/\kappa}(y_i) = \bar{u}$ 
  
  iii)  $y_1  \in N_{\kappa_1(   \sqrt{\bar{w}})/\kappa_1}(\kappa_1(  \sqrt{\bar{w}})^* ) N_{\kappa_1(  \sqrt{\bar{d}\bar{w}})/\kappa_1}
  ( \kappa_1(  \sqrt{\bar{d}\bar{w}})^*)$. 
  
  Let $\nu$ be a place of $\kappa$. Suppose that $\nu \in S_0$.
  Then, by the choice, $[\kappa_\nu(\sqrt{x_{\nu}}), \kappa_{\nu}] = 2$. Since $\kappa_{\nu}$ is a local field, the $2$-torsion Brauer subgroup of $Br(\kappa_v)$ is isomorphic to $\mathbb{Z}/2\mathbb{Z}$ and 
  $(\bar{b}, \bar{c})$ is split over $\kappa_\nu (\sqrt{x_{\nu}}) = \kappa_\nu(\sqrt{x})$ (cf. \cite[Chapter 31, Theorem 4, Theorem 5]{algebraII}).
  Suppose that $\nu \not\in  S_0$. Then, by the choice of $S_0$,  $b$ and $c$ are units at $\nu$. Since $\bar{c}$ is a norm of the unramified extension $\kappa_v(\sqrt{\bar{b}})/\kappa_v,$ 
     $(\bar{b}, \bar{c})$ is split  over $\kappa_\nu$ (cf. \cite[Chapter 31, Theorem 2]{algebraII}).
  Hence, $(\bar{b}, \bar{c})$ is split  over $\kappa_1 =  \kappa(\sqrt{x})$ (cf. \cite[Corollary 6.5.3]{GS2017}).

  Let $a \in R_0$ be a lift of $x$, 
  $\mu_1 \in R_0[X]/(X^2 - a)$  a lift of $y_1$ and  $\mu_2  \in R_0[X]/(X^2 - w)$  a lift of $y_2$. 
  Since $R_0$ is complete, there exists $\mu_3 \in  R_0[X]/(X^2 - wd)$ which has a lift $y_3$
  such that $\prod_i N_{L_i/F_0}(\mu_i) = u$.
  
 Since $R_0$ is complete and $(\bar{b}, \bar{c})$ splits over $\kappa(\sqrt{x})$, 
 $(b, c) \otimes L_1$ is split (cf. \cite[Corollary 6.8.8]{GS2017}).
 
 Since $y_1 \in N_{ \kappa_1(  \sqrt{\bar{w}}, \sqrt{\bar{d}})/\kappa_1( \sqrt{\bar{d}})}
( \kappa_1( \sqrt{\bar{w}}, \sqrt{\bar{d}})^*) $, $R_0$ is complete and $\mu_1$ is a lift of $y_1$,  
  $\mu_1 \in N_{ L_1(  \sqrt{ w}, \sqrt{ d })/L_1  ( \sqrt{ d})}
(L_1(  \sqrt{ w}, \sqrt{ d})^*)$.  Since $(b, c) \otimes L_1 $ is split, 
$D\otimes L_1 = (w,\pi) \otimes L_1(\sqrt{d})$. Hence $\mu_1$ is a reduced norm from $D\otimes L_1$. 

Let $i = 2, 3$.  
Since $\kappa$ is a global field of positive characteristic, $y_i$ is a reduced norm from $(\bar{b}, \bar{c}) \otimes \kappa_i$ (cf. \cite[Theorem 33.5]{reiner}).
Since $R_0$ is complete, $\mu_i $ is a reduced norm from $(b, c) \otimes L_i(\sqrt{d}) = D\otimes L_i $. 

 Hence $a_1, \mu_1, \mu_2, \mu_3$ have the required properties.    
      
  \end{proof}

\section{Two dimensional complete fields}\label{section5}

Let $R$ be a complete regular local ring with maximal ideal $(\pi, \delta)$,
field of fractions $F$ and residue field $\kappa$.  Suppose that  
char$(\kappa)$ not equal to 2. 
Let $F_\pi$ be the completion of $F$ at the discrete valuation given by $(\pi)$.

\begin{lemma} 
\label{branch-elements} 
Let $a = u \pi^{\epsilon}\delta^{\epsilon'} \in R$ with $u \in R$ a unit,  $\epsilon, \epsilon' \in \Z$ and 
$R'$ the integral closure of $R$ in $F(\sqrt{a})$. 
Then for any element $\mu_\pi \in F_\pi(\sqrt{a})^*$, there exists $\mu \in F(\sqrt{a})^*$ such that 
$\mu = w \pi^r \delta^s \sqrt{a}^{s'} $ with $w \in R' $  a unit, $r, s, s' \in \Z$ and  $\mu_\pi\mu^{-1}   \in F^{*2}_\pi$.   
\end{lemma}

\begin{proof}  Suppose that $\epsilon$ and $\epsilon'$ are even.  Then $F(\sqrt{a}) = F(\sqrt{u})$.
Since $u$ is a unit in $R$ and char$(\kappa) \neq 2$, 
  $R' = R[\sqrt{u}[$   is a regular local ring with  the maximal ideal $(\pi,\delta)R'$.
Hence     $ \mu_\pi = w\pi^r \delta^s c^2$ for some unit $u \in R'$, $c \in F_\pi(\sqrt{a})$
and $r, s \in \Z$ (cf. \cite[Remark 7.1]{LocalglobalPF}). Hence $\mu=   w\pi^r\delta^s$ has the required property.

Suppose $\epsilon$  is odd.  Write $\epsilon = 2t + 1$ for some $t \in \Z$.
 Then $F_\pi(\sqrt{a}) = F_\pi(\sqrt{u\pi \delta^{\epsilon'}})$ is a totally ramified extension. 
 Further the residue field of $F_\pi(\sqrt{a})$ is  equal to the residue field 
 of $F_\pi$ and hence isomorphic to the field of fractions of $R/(\pi)$. 
 Since     $\sqrt{u\pi \delta^{\epsilon'}} = \sqrt{a}/\pi^t$ is a parameter in $F_\pi(\sqrt{a})$, 
 $\mu_\pi = \theta  (\sqrt{a}/ \pi^t)^{s'}$ for some $\theta \in F_\pi(\sqrt{a})$ a unit in the valuation ring. 
 Let $\bar{\theta}$ be  the image of $\theta$  in the residue field of $F_\pi(\sqrt{a})$. 
 Since the residue field of $F_\pi(\sqrt{a})$ is the field of fractions of $R/(\pi)$ and 
 $R/(\pi)$ is a complete discrete valuation ring,  we have $\bar{\theta} = \bar{u} \bar{\delta}^s$ for some $u \in R$ a unit and $s \in \Z$. 
 Let $\mu = u \delta^s (\sqrt{a}/ \pi^t)^{s'} = u \delta^s \pi^{-ts'} (\sqrt{a})^{s'}.$
 Then $\mu_\pi \mu^{-1} $ is a unit  in the valuation ring of $F_\pi(\sqrt{a})$ and  its image in the residue field is 1.
 Since char$(\kappa) \neq 2$, $\mu_{\pi} = \mu c^2$ for some $c \in F_\pi(\sqrt{a})$ and hence 
 $\mu$ has the required property.    
\end{proof}

\begin{lemma}
\label{curve-points0} 
 Let $\lambda = u\pi^r\delta^s \in F^*$ with $u \in R^*$.
Let  $n \geq 1$, $a_{i\pi} \in F_\pi^*$ and $\mu_{i\pi} \in  L_{i\pi} = F_{\pi}[X]/(X^2 -a_{i\eta})$ 
for $1 \leq i \leq n$
with  $\prod_i N_{L_{i\pi}/F_\pi}(\mu_{i\pi}) = \lambda$.
Then there exist $a_i  =  u_i \pi^{\epsilon_i}\delta^{\epsilon_i'} \in R$ with $u_i \in R^*$,
$\mu_i  = w_i \pi^{r_i}\delta^{s_i} \sqrt{a_i}^{s_i'}$ for some $w_i \in R[X]/(X^2 - a_i)^*$
such that

i) $a_{i\pi}a_i^{-1} \in F_\pi^{*2}$ for all $i$

i)  $\prod_i N_{F[X]/(X^2 -a_i)/F}(\mu_i) = \lambda$ 

ii) there is an isomorphism $\phi_i : F_\pi[X]/(X^2 - a_{i\eta}) \to F_\pi[X]/(X^2 - a_i)$
with $\phi_i(\mu_{i\pi})\mu_i^{-1} \in F_\pi[X]/(X^2 - a_i)^{2^m}$ for all $m \geq 1$. 
\end{lemma}

\begin{proof} Applying (\ref{branch-elements}) for $a_{i\pi}$ with $a = 1$,
we get $a_i = u_i \pi^{\epsilon_i} \delta^{\epsilon_i'}$ with $u_i \in R^*$ such that 
$a_i a_{i\pi} \in F_\pi^{*2}$.  Hence replacing $a_{i\pi}$ by $a_{i}$ we assume that 
$\mu_i \in F_\pi[X]/(X^2 - a_i)$.  

Let $1 \leq i \leq n $. Suppose $a_i$ is a square in $F$. 
Then $F_\pi[X]/(X^2 - a_i) = 
 F_\pi \times F_\pi$ and $\mu_{i\pi} = (\mu'_{i\pi}, \mu''_{i\pi})$.
  Let  $\mu'_i, \mu''_i \in F$ be as in (\ref{branch-elements}) corresponding to
  $\mu_{i\pi}'$ and $\mu_{i\pi}''$ and $\mu_i = (\mu'_i, \mu_i'') \in F[X]/(X^2 - a_i)$. 
Suppose $a_i$ is not a square. Let  $\mu_i  \in F(\sqrt{a_i})$ be as in (\ref{branch-elements}) corresponding to
  $\mu_{i\pi}$.  Then, for each $i$, $N_{F[X]/(X^2 - a_i)/F}(\mu_i)$ is a product of a unit  in $R$ and powers of $\pi$ and $\delta$.
  Hence $\prod_i N_{F[X]/(X^2 - a_i)/F}(\mu_i)$ is a product of a unit in $R$  and powers of $\pi$ and $\delta$.
  Since $\mu_{i\pi} = \mu_i c_i^2$ for some $c_i \in F_\pi(\sqrt{a_i})$ 
  and $\prod_i N_{L_{i\pi}/F_\pi}(\mu_{i\pi}) = \lambda$, 
 $ \lambda^{-1}\prod_i N_{F[X]/(X^2 - a_i)/F}(\mu_i)$  is a  square in $F_\pi$.  
 Since  $ \lambda^{-1}\prod_i N_{F[X]/(X^2 - a_i)/F}(\mu_i)$ is a product of a unit in $R$ and powers of $\pi$ and $\delta$, 
  there exists   $\theta \in F^*$  such that  
  $\lambda^{-1}\prod_i N_{F[X]/(X^2 - a_i)/F}(\mu_i) = \theta^2$ (cf. \cite[Remark 7.1]{LocalglobalPF}).
  Replacing $\mu_1$ by $\mu_1\theta^{-1}$, we have the required $\mu_i$.
\end{proof}

The following is well known.

\begin{lemma} 
\label{2dim-quat} 
Suppose that  $\kappa$ is a finite field.  Then $D = (v, \pi)$ or $(v, \delta)$ or  $(v, \pi\delta)$ or $(v_1\pi,  v_2\delta)$  for some units 
$v, v_1, v_2 \in R$.
\end{lemma} 

\begin{proof} By  (\cite[Theorem 2.1]{salt}),  we have 
$D = D_0 +  (u,\pi) + (v, \delta)$ or $D_0 + (u\delta, v\pi)$, where   $D_0$ is unramified and  $u, v \in R^*$ are units.  
Since $\kappa$ is finite field and $D_0$ is unramified, $D_0  = 0$. Hence $D =    (u,\pi) + (v, \delta)$ or $  (u\delta, v\pi)$  for some 
units $u, v \in R^*$.  

Suppose $D  =    (u,\pi) + (v, \delta)$ with $u, v \in R^*$.  If $u$ are $v$ is a square in $R$, then $D = (u, \pi)$ or $(v, \delta)$.
Suppose $u$ and $v$ are not  squares. Since $\kappa$  is a finite field and $R$ is complete, $uv$ is a  square in $R$.
Hence $D = (u, \pi) + (u, \delta)  =  (u, \pi\delta)$. 
\end{proof}

\begin{lemma}
\label{curve-points0-norms} 
Suppose that  $\kappa$ is a finite field.    Let $D$ be a quaternion algebra over 
$F$ which is unramified on  $R$ except possibly at $(\pi)$ and $(\delta)$.
Let $a  =  u \pi^{\epsilon}\delta^{\epsilon'} \in R$ with $u \in R^*$
and $\mu  = w  \pi^{r}\delta^{s} \sqrt{a}^{s'} \in F(\sqrt{a})$  for some $w$ a unit in the integral closure of 
$R$ in $F(\sqrt{a})$.   If $\mu$ is a reduced norm from $D\otimes F_\pi(\sqrt{a})$, 
then $\mu$ is a reduced norm from $D\otimes F(\sqrt{a})$. 
\end{lemma}

\begin{proof}  Let  $L = F(\sqrt{a})$ and $S$  the integral closure of $R$ in $L$. 

Suppose $\epsilon$ or $\epsilon'$ is  even. Then $S$ is a 
regular  local ring with maximal ideal $(\pi', \delta') $ for  primes $\pi'$ and $\delta'$
lying over $\pi$ and $\delta$.  Since $D$ is unramified on $R$ except possible at $(\pi)$ and $(\delta)$, 
$D\otimes_F L$ is unramified on $S$ except possible at $(\pi')$ and $(\delta')$.
Hence, the result follows from \cite[Corollary 7.3]{LocalglobalPF}.

Suppose both $\epsilon$ and $\epsilon'$ are odd.
Then $a = u\pi\delta a_1^2 $ for some $a_1 \in F^*$. 
Without loss of generality, we assume that $a = u\pi\delta$. 

Suppose $D\otimes L $ is split.  Then every  element is a reduced norm from $D\otimes L  $.
Hence we assume that  $D\otimes L $ is non split.

Since $\kappa$ is a finite field, we have 
$D = (v, \pi)$ or $(v, \delta)$ or  $(v, \pi\delta)$ or $(v_1\pi,  v_2\delta)$  for some units 
$v, v_1, v_2 \in R$ ( cf. \ref{2dim-quat}).

 Since $\kappa$ is a finite field, $Br(R)\cong Br(\kappa)=0$ and $(w_1, w_2 )$ is split for all units $w_1, w_2  \in R^*$ (cf. \cite[Lemma 5.3]{PPS}). Since $(v, \pi\delta) \otimes   L  = (v, u) \otimes L$  is split, 
 $ D \neq (v, \pi\delta)$.    Since $ (v_1\pi, v_2\delta) \otimes L   = (v_1 v_2 u, \delta ) \otimes L $, 
 without loss of generality we assume that $D = (v, \pi)$ or $(v, \delta)$ for some unit $v \in R^*$.  
Since $(v, \pi) \otimes L  = (v, \delta) \otimes L  $, we assume that 
$D = (v, \delta)$. 

Since $a = u\pi \delta$,  $\sqrt{a}$ is a parameter in $F_\pi(\sqrt{a})$. 
Suppose that $\mu$ is a reduced norm from $D\otimes F_\pi(\sqrt{a})$.  
Since $D \otimes F_{\pi}(\sqrt{a})$ is unramified and $\sqrt{u\pi\delta}$ is a parameter in
$F_{\pi}(\sqrt{a})$, it follows that $s'$ is even, because the cup product
$\partial(\mu) \cup [D]$ vanishes in
$
H^3\big(F_{\pi}(\sqrt{a}), \mu_2^{\otimes 2}\big)$ (cf. \cite[Theorem 8.9.1]{GS2017}, \cite[section 4]{PPS}).
  Since  $D \otimes L  = (v, \delta) \otimes L  =  (v, \pi) \otimes L$,
$-\pi$ and $-\delta$ are reduced norms from $D$. Hence it is enough to show that 
$\pm w$ is a reduced norm from $D\otimes L$.  Since $w$ is a unit in $R$ and 
the residue field of $R$ is finite, the  quaternion algebra $(v, \pm w)$ is split over $L$ by the same reason as before, namely that $Br(\kappa)=0$.
Hence $\pm w$  is a reduced norm from $L(\sqrt{v})$. Since $L(\sqrt{v}) $ is a maximal subfield of 
$D\otimes L$, $\pm w$ is a reduced norm from $D\otimes L$.
\end{proof}

 Let  $R_0$ be a complete two dimensional regular local ring with   $m = (\pi, \delta)$ the maximal ideal of $R_0$,
 $\kappa_0 = R/m$ and $F_0  $ field of fractions of $R_0$.   Suppose that 
  $\kappa_0$ a finite field of char not equal to 2. 
 Let  $w \in R_0$  be a unit which is not a square in $R_0$.
 Then every unit  $u$ in $R_0$ is either a square in $R_0$ or $ uw$ is a  square in $R_0$.  Let  
  $d = w$ or  $\pi$.  Let  $F = F_0(\sqrt{d})$ and $R$ the integral closure of $R_0$ in $F$.  
  Then $R$ is a regular local ring with maximal ideal $m_R = (\pi', \delta)$  with $\pi' = \pi$ if $d = w$ and 
  $\pi' = \sqrt{\pi}$ if $d = \pi$ (\cite[Lemma 3.1, Lemma 3.2]{uinv}). Let $D/F$ be a central  division algebra of  period 2 which is
unramified on $R$ except possibly at 
$(\pi')$ and $(\delta)$.

\begin{prop}
\label{2dim-class} 
Suppose that $D$ admits a $F/F_0$-involution.  
Then there exists a quaternion  division algebra  $D_0/F_0$ such that 
$D \simeq  D_0\otimes F$  and \\
i) if $d = w$, then $D_0 = (\pi, \delta)$ \\
 ii) if $d = \pi$, then   $D_0 = (w, \delta)$. 
\end{prop} 

 \begin{proof} Suppose $d = w$.  Then $R$ is a   complete regular local ring with maximal ideal $(\pi,\delta)$.
 Hence,  by (\ref{2dim-quat}), we have $D = (u, \pi)$ or $(u, \delta)$ or $(u, \pi\delta)$ or $(u\pi, v\delta)$ for some units 
 $u, v \in R^*$.  
 
  Suppose $D= (u, \pi a)$ for some unit $u \in  R^*$ and $a \in F^*$.    Since $D$ has a $F/F_0$-involution, 
  cores$_{F/F_0}(D)=(N_{F/F_0}(u),\pi a)   = 0 \in Br(F_0)$ (cf. \cite[Chapter 1, 3.B.]{knus}). Since $D$ is not split and  $\pi a$ is a parameter of $F,$ cores$_{F/F_0}(u) = N_{F/F_0}(u) = 1$.
Thus there exists $\theta \in R^*$ such that $u  = \theta \bar{\theta}^{-1}$. Since $u = \theta\bar{\theta} \bar{\theta}^{-2}$ 
  and $\theta\bar{\theta} \in R_0^*$,
  we can replace $u$ by $\theta$ and assume that $u  \in R_0^*$. Since $\kappa_0$ is a finite field and $w \in R_0^*$ not a square, 
  $u$ is a square in $R$. Hence $D = 1$. Since $D$  is a quaternion division algebra, $D \neq  (u, \pi)$. and $D \neq (u, \pi\delta)$. 
  
  Suppose $ D = (u, \delta)$. Since $D$ has a $F/F_0$-involution, cores$_{F/F_0}(D) = 0$.
  Since $\delta \in F_0$, we have cores$_{F/F_0}(u, \delta) = (N_{F/F_0}(u), \delta)$ (The reference on the corestriction of a cyclic algebra is mentioned in the proof of \ref{dvr-class-4}).
  Since $\delta$ is a regular parameter in $R_0$,  $N_{F/F_0}(u)$ is a unit in $R_0$ and 
  $(N_{F/F_0}(u), \delta) = 1$, it follows that $N_{F/F_0}(u)$  is a square in $R_0$. Thus 
  $N_{F/F_0}(u) = v^2$ for some $v \in R_0^*$. Then $N_{F/F_0}(uv^{-1}) = 1$. Hence, as above, 
  $u = vv_1$ for some $v_1 \in R_0^*$ and hence $u$ is a square in $R$. Thus $D \neq (v, \delta)$.
  
  Hence $D = (u\pi , v\delta)$ for some units $u, v \in R^*$. As above  we can choose $u, v \in R_0^*$.

  Suppose $d = \pi$.  Then  $R$ is a  complete  regular local ring with maximal ideal $(\sqrt{\pi}, \delta)$.
  Hence, by (\ref{2dim-quat}), $D = (u, \sqrt{ \pi} )$ or $(u, \delta)$ or $(u, \sqrt{\pi} \delta)$ or $(u\sqrt{\pi}, v\delta)$ for some units 
 $u, v \in R^*$.   Since the residue fields of $R$ and $R_0$ are  isomorphic, we can assume that $u, v \in R_0^*$. 
 Suppose $D = (a, \sqrt{\pi} b )$ for some $a, b  \in R_0$ not divisible by $\sqrt{\pi}$. Since cores$_{F/F_0}(a, b) = 2(a, b) = 0$, 
 cores$_{F/F_0} (D) = (a, -\pi)$. Since cores$_{F/F_0}(D) = 0$, $(a, -\pi) = 0$. By taking the residue at $\pi$, we see that 
 the image of $a$ in the residue field at $(\pi)$ is  a square.  Hence $a \neq u$ and $a \neq   v\delta$.
 Thus $D = (u, \delta)$. Since $D$ is a division algebra, $u$ is not a square in $R_0$. Hence $D = (w, \delta)$.  
     \end{proof}

Let   $\lambda = u \pi^r\delta^s \in F_0$ with $u \in R_0$ a unit and $r,s \in \Z$. 
Suppose that $\lambda$ is a reduced norm from $D$.
In this section we construct quadratic extensions $L_1, L_2, L_3$  of $F_0$ and $\mu_i \in L_i$ with 
$\prod_i N_{L_i/F_0} (\mu_i) = \lambda$ and satisfying  some other properties. 
These results are used in the proof of the main theorem.  Let $\epsilon_1, \epsilon_2 \in\{0, 1 \}$
such that $r = 2r_1 + \epsilon_1 $ and $s =  2s_1+\epsilon_2$ for some $r_1, s_ 1 \in {\mathbb Z}$. 
Then $\lambda = u \pi^{\epsilon_1} \delta^{\epsilon_2} (\pi^{r_1} \delta^{s_1})^2$.

\begin{prop} 
\label{2dim-square}
Let $D_0$ be a quaternion division algebra over $F_0$ which is unramified on 
$R_0$ except possibly at $(\pi)$ and $(\delta)$. Suppose that $\lambda $ is a reduced norm from $D_0$. 
Then  there exist  $a_i  \in F_0^*$ 
  and $\mu_i  \in L_i = F_0[X]/(X^2 - a_i)$  for $i = 1, 2$ such that 
  
i)  $a_1 = v\pi^{\epsilon_1} \delta^{\epsilon_2}
  \in F_0^* \setminus F_0^{*2}$ with   $v\in R_0^*;$ 

ii) $a_2 \in R_0$,  $a_2$ is a unit at $(\pi)$  and  $(\delta)$,  
$\partial_{\pi}(D_0) = \bar{a}_2 \in\kappa(\pi)^*/\kappa(\pi)^{*2}$   and
$\partial_{\delta}(D_0) = \bar{a}_2 \in\kappa(\delta)^*/\kappa(\delta)^{*2}.$
  
iii)  $ N_{L_1/F_0}(\mu_1) \cdot N_{L_2/F_0}(\mu_2)= \lambda;$

iv)  $\mu_i$ is a reduced from from $ D_0 \otimes  L_i$ for $i = 1, 2;$

v) $\mu_2$ is a unit in $R_0[X]/(X^2 - a_2).$ 
\end{prop}

\begin{proof}   Since $\kappa$ is a finite field, we have  (\cite[Lemma 3.6]{Wu}, cf. \ref{2dim-quat} ) 
$D_0 = (v, \pi)$, $(v, \delta)$, $(v, \pi\delta)$ or
$(v_1\pi, v_2\delta)$ for some units $v, v_1, v_2 \in R_0$. 
If $D_0 = (v, \pi)$, let $a_2 = v\delta^2 + \pi^2$.
If $D_0 = (v, \delta)$, let $a_2 = v\pi^2 + \delta^2$. 
If $D_0 = (v, \pi\delta)$, let $a_2 = v$.  
If $D_0 = (v_1\pi,  v_2\delta)$, let $a_2 = v_1\pi + v_2\delta$.   
Then $a_2$  satisfies the condition ii).

Suppose $\pm \lambda$ are not squares in $F_0$. 
Let $a_1   = - u\pi^{\epsilon_1} \delta^{\epsilon_2}$ and 
 $\mu_1 =  \pi^{r_1} \delta^{s_1} \sqrt{a_1} $. 
 Then $N_{L_1/F_0}(\mu_1) = \lambda$ and  by (\cite[Lemma 6.2]{PPS}),   $\mu_1$  is a reduced norm from $D_0 \otimes L_1$. 
Hence   $a_1, a_2, \mu_1,$ and $\mu_2 = 1$ have the required properties.
 
 Suppose one of $\pm\lambda$ is a square in $F_0$. 
 Then  $\epsilon_1 = \epsilon _2 = 0$, $u = \pm  u_1^2$
for some $u_1 \in R_0^*$ and $\lambda =  \pm u_1^2 \pi^{2r_1} \delta^{2s_1}$. 
Let $a_1 =  w \in R_0^*$ not a square. Then $a_1$ satisfies the condition i). 
 
Suppose  $\lambda$ is a square. Then $\lambda   =  u_1^2 \pi^{2r_1} \delta^{2s_1}$. 
Suppose $D_0 = (v_1\pi, v_2\delta)$.
Let    $\mu_1 =  (-1)^{r_1 + s_1} v_1^{r_1}v_2^{s_1} \pi^{r_1}\delta^{s_1}$
and $\mu_2 = u_1 v_1^{-r_1}v_2^{-s_1}$.   Then $N_{L_1/F_0}(\mu_1) N_{L_2/F_0}(\mu_2) = \mu_1^2\mu_2^2 =  \lambda$. 
Since $-v_1\pi$ and $-v_2\delta$ are reduced norms from $D_0$,  
$\mu_1$ is a reduced norm from $D_0$. 
Let $x, y \in D_0$ such that $x^2 = v_1\pi$, $y^2 = v_2\delta$ and $xy = -yx$. Then 
 $Nrd(x + y)  = -(x^2 + y^2) = -(v_1\pi + v_2\delta) = -a_2$. In this case, $v_2\delta$ is a norm of the extension $L_2F_0(\sqrt{v_1\pi})/L_2.$ Hence, $D_0 \otimes L_2$ is split. In particular $\mu_2$ is a reduced norm from $D_0 \otimes L_2$. 
Thus  $a_1, a_2, \mu_1$ and $\mu_2$ have the required properties.

Suppose $D_0 \neq (v_1\pi, v_2\delta)$.  Since $a_1  = w \in R_0^*$ a non-square and 
$\kappa$ is a finite field, every  unit in $R_0$ is a square in $L_1$.  Thus  $D_0 \otimes L_1$ is split. 
Hence $a_1, a_2,  \mu_1 = u_1\pi^{r_1}\delta^{s_1}$ and $\mu_2 = 1$ have the required properties. 

Suppose $\lambda$ is not a  square. Then $\lambda  = -u_1^2 \pi^{2r_1} \delta^{2s_1}$ and $-1 $ is  not a square $F_0$. 
 Since $\lambda$ is a reduced norm from, $-1$ is a reduced norm from $D_0$. Since $-1$ is not a square in  $F_0$,  
  it follows that $D_0 \neq (v_1\pi, v_2\delta)$.  Hence $D_0 \otimes L_1$ is split. 
 Since $a_1 =  w \in R^*$  and $\kappa$ is a finite field, the quaternion algebra $(-1, a_1)$ is  split over $F_0$.
  Hence,  there exists $\mu' \in  L_1 = F_0(\sqrt{a_1}) $ such that 
 $N_{L_1/F_0}(\mu') = -1$ (cf. \cite[Lemma 2.4]{bh}). 
 Then  $a_1$,  $\mu_1 = 
  \mu' u_1\pi^{r_1}\delta^{s_1}$,
and $\mu_2 = 1$  have the required properties.  
\end{proof}

\begin{lemma}
\label{2dim-unit} Suppose that  $d = w$.  Let $a_2 =   \pi + \delta$ and $a_3 = d a_2$. 
 There exist $a_1 = v\pi^{\epsilon_1} \delta^{\epsilon_2}
  \in F_0^* \setminus F_0^{*2}$  with $v\in R_0^*$
  and $\mu_i \in L_i = F_0[X]/(x^2 - a_i)$ such that 
  
i)  $\prod_{i =1}^3 N_{L_i/F_0}(\mu_i) = \lambda$

ii)  $\mu_i$ is a reduced from $ D \otimes _F FL_i$ for $i = 1, 2, 3$

iii)  $\mu_i \in L_i$ is a  unit  at $\pi$ and $\delta$ for $i =  2, 3$.   
\end{lemma}

\begin{proof}
Since $F = F_0(\sqrt{d})$, every element in $R_0^*$ is a square in $R$. 
In particular $-1 \in R^{*2}$ and $\lambda =  u_1^2 \pi^r\delta^s$ for some $u_1\in R$ with $u_1^2 = u$. 
Further $D_0 = (\pi, \delta)$ (\ref{2dim-class}). 

Suppose     $  \lambda$  is  not  a square in $F^*$.  Then 
$\lambda = \pi^{\epsilon_1} \delta^{\epsilon_2} u_1^2\pi^{2r_1} \delta^{2s_1}$.
Since $-1 \in F^{*2}$,  both $\pm \lambda$ are not squares in 
$F^*$.  Hence, by \cite[Lemma 6.2]{PPS}, $a_1 = - u\pi^{\epsilon_1} \delta^{\epsilon_2}$, 
 $\mu_1 =  \pi^{r_1} \delta^{s_1} \sqrt{a_1} $ and  
$\mu_2 = \mu_3 = 1$ have the required properties.

Suppose that   $ \lambda$ is   a square in $F^*$. 
Then $r = 2r_1$ and $s = 2s_1$. 
Suppose  $\lambda$ is a  square in $F_0$. 
Then $u_1 \in R_0^*$. 
Since  $\pi$ and $\delta$ are reduced norms from $D_0$ and $D_0\otimes F_0(\sqrt{a_2})$ is split,    
  $ a_1 = w$,  $\mu_1 =  \pi^{r_1}\delta^{s_1}$,  $\mu_2 =  u_1$ and $ \mu_3 = 1$ have the required properties. 

Suppose that $\lambda \not\in F_0^{*2}$. Then   $u = wu_2^2$ for some  unit $u_2 \in R_0$.   
  Then $a_1 = w$,  $\mu_1 =  \pi^{r_1}\delta^{s_1}$, $\mu_2 = \sqrt{a_2}^{-1}$, 
  $\mu_3 =  u_2 \sqrt{a_3}$ have the required properties.  
\end{proof}

\begin{lemma}
\label{choice-at-P-pi} Suppose that  $d = \pi$. 
There exist $a_1 =v\pi^{\epsilon_1} \delta^{\epsilon_2} \in 
F_0^* \setminus F_0^{*2}$  with $v\in R_0^*$   and $\mu_1 \in L_1 = F_0[X]/(x^2 - a_1)$ and 
$\mu_2 \in  L_2 = F_0[X]/(x^2 - w)$  such that 

i)  $  N_{L_1/F_0}(\mu_1)  N_{L_2/F_0}(\mu_2)= \lambda$

ii)  $\mu_i$ is a reduced from from $ D \otimes _F FL_i$ for $i = 1, 2$

iii) $ \mu_2 \in R_0[X]/(X^2 - w)$ a unit. 
\end{lemma}

\begin{proof} Since $d= \pi$, we have $D_0  = (w, \delta)$ (\ref{2dim-class}). 

Suppose   both  $\pm \lambda$  are   not  squares  in $F$. 
Then  $a_1 = -u\pi^{\epsilon_1} \delta^{\epsilon_2} $, $\mu_1 = \pi^{r_1} \delta^{s_1} \sqrt{a_1} $ and  
$\mu_2  =  1$ have the required properties (\cite[Lemma 6.2]{PPS}).

Suppose that  only one  of $\pm \lambda$ is   a square in $F$. 
Then  $-1 \not\in F_0^{*2}$ and $\lambda = \pm 1  \in F_0^*/F_0^{*2}$ or $ \lambda = \pm \pi \in F_0^*/F_0^{*2}$.
Further   $D_0  = (-1, \delta)$. 

Suppose  $\lambda = \pm 1 \in F_0^*/F_0^{*2}$.  Then $\lambda = \pm u_1^2 \pi^{2r_1} \delta^{2s_1}$ 
for some $u_1 \in R_0^*$. 
Let $a_1 = -1$  
Since $\kappa$ is a finite field, there exists 
$\mu'_1 \in L_1 = F_0(\sqrt{a_1})$ with $N_{L_1/F_0}(\mu_1)= \pm 1$.
Then $a_1$, $\mu_1 = u_1\pi^{r_1}\delta^{s_1} \mu_1'$, 
$\mu_2 =   1$ have the required properties.   

 Suppose  $\lambda = \pm \pi \in F_0^*/F_0^{*2}$. 
 Then $\epsilon_1 = 1$, $\epsilon_2 = 0$ and 
 $\lambda = \epsilon \pi  u_1^2 \pi^{2r_1} \delta^{2s_1}$ 
for some $u_1 \in R_0^*$ and $\epsilon = \pm 1$. 
 Then  $a_1 = -  \pi$, 
 $\mu_1 =  u_1\pi^{r_1}\delta^{s_1}\sqrt{-  \pi}$ and  $\mu_2 \in L_2 = F_0(\sqrt{-1})$ with 
 $N_{L_2/F_0}(\mu_2) =  \epsilon$  
 have the required properties.

Suppose both  $\pm \lambda$ are  squares in $F$. 
Then $-1 \in F^{*2}$. Since $d = \pi$, $-1 \in F_0^{*2}$. 

Suppose  $\lambda$ is a  square in $F_0$. 
Then  $a_1 = w$, $\mu_1 = \sqrt{\lambda}$ and  $\mu_2   = 1$ have the required properties.

Suppose that $\lambda$ is not a square in $ F_0^ *$. Then  
  $d\lambda \in F_0^{*2}$  and hence $\lambda = \pi  u_1^2 \pi^{2r_1} \delta^{2s_1}$ 
for some $u_1 \in R_0^*$.
  Then $a_1 =  w\pi$, $\mu_1 =  u_1\pi^{r_1}\delta^{s_1}\sqrt{w\pi}$ and $\mu_2 \in L_2$ with 
  $N_{L_2/F_0}(\mu_2)  = -w^{-1}$ 
  have the required properties. 
\end{proof}

\begin{lemma}
\label{curve-point}
Suppose $d$ is not a square and  $D$ is ramified on $R$ at most at $\pi$. 
Then $D$ is split. 
\end{lemma} 

\begin{proof}  Suppose $d = w$.
Suppose  $D$ is non split.
Then, by (\ref{2dim-class}), 
$D \simeq (\pi, \delta) \otimes F$.  Then $D$ is ramified both at $(\pi)$ and 
$(\delta)$. This contradicts the assumption that $D$ is  ramified at most  at $\pi$. 
Hence $D$ is split. 

Suppose $d = \pi$. 
Suppose  $D$ is non split.
Then, by (\ref{2dim-class}), 
$D \simeq (w, \delta) \otimes F$ for unit $w \in R_0$ which is not a square. 
Since $F/F_0$ is ramified, $w\in R$ is not a square. In particular 
$D$ is ramified at $\delta$.  This contradicts the assumption that $D$ is  ramified at most  at $\pi$. 
Hence $D$ is split. 
\end{proof}

We end this section with the following. 

\begin{prop} 
\label{curve-point-choice0}  Suppose that $D$ is ramified on $R$ at most at $\pi$. 
Let $n \geq 1$. 
Suppose there exist $a_{i\pi} \in F_{0\pi}$ and 
$\mu_{i\pi} \in L_{i\pi} = F_{0\pi}[X]/(X^2 - a_{i\pi})$  for $1 \leq i \leq n$  such that 

i) $\prod_i N_{L_{i\pi}/F_{0\pi}}(\mu_{i\pi}) = \lambda$

ii) $\mu_{i\pi}$ is a reduced norm from $D\otimes L_{i\pi}$  for $1 \leq i \leq n$.

Then there exist $a_{i} \in F_{0}$ and $\mu_{i} \in L_{i} = F_{0}[X]/(X^2 - a_{i})$ for 
$1 \leq i \leq n$
such that 
i) $\prod_i N_{L_{i}/F_{0}}(\mu_{i}) = \lambda$

ii) $\mu_{i}$ is a reduced norm from $D\otimes L_{i}$  for $1 \leq i \leq n$

iii) $a_{i\pi}a_{i} \in F_{0, \pi}^{*2}$ for $1 \leq i \leq n$

iv)  there is an isomorphism 
$$\phi_i : L_{i\pi} \simeq L_{i} \otimes F_\pi$$ such that 
$\phi_i(\mu_{i\pi})^{-1} \mu_{i}  \in ( L_i \otimes F_{0\pi})^{2^m}$  for all $m \geq 1$ and
 $1 \leq i \leq n$. 
\end{prop}

\begin{proof} Apply   (\ref{curve-points0}) to  $R_0$, $F_0$ and $a_{i\pi}$ with $a = 1$,
and get $a_i \in F_0$ as in  (\ref{curve-points0}).   Apply  once again (\ref{curve-points0})   to
 $R_0$,  $F_0$,  $a_i$ and 
$\mu_{i\pi}$,  get $\mu_i \in  L_i = F_0[X]/(X^2 - a_i)$ as in   ( \ref{curve-points0}). 

Suppose $d$ is not a square in $F_0$.  Then, by (\ref{curve-point}), $D$ is split and hence 
$\mu_i$ are reduced norms from $D$.

Suppose $d$ is a square in $F_0$.  Then $F = F_0 \times F_0$ and $D = 
D_0 \otimes F = D_0 \times D_0$. 
Since $\mu_{i\pi}$ are reduced norms from 
$D\otimes_F FL_i   =  D_0 \otimes L_i   \times D_0 \otimes L_i$.
Hence $\mu_{i\pi}$ are reduced norms from $D_0 \otimes L_i$ and 
by (\ref{curve-points0-norms}), $\mu_i$ are reduced norms from $D\otimes L_i$. 
\end{proof}

\section{Choices at nodal points}\label{section6} 
\label{nodalpoints}

Let  $p \geq  3$ be  a prime and  $K$ be a  $p$-adic field.
Let $F_0$ be the function field of a curve over $K$ and 
$F = F_0(\sqrt{d})$ a quadratic field extension.  
Let $D$ be a central division algebra over $F$ with a $F/F_0$-involution. 
Let $\lambda \in F_0^* \cap Nrd(D)^*$. 

Let $T$ be the  valuation ring of $K$ and $k$ the residue field of $K$.
Let $\XX_0$ be a regular proper model of $F_0$ over $T$ with  the union of the 
ramification locus of $D$, support of $d$, support of $\lambda$
 and the closed fibre $X_0$ of $\XX_0$   is a union of regular 
curves with  normal crossings. Further, the integral closure $\XX$
of $\XX_0$ in $F$ is a regular proper model of $F$ (cf. \cite[Theorem 11.2]{LocalglobalPF}, or \cite{alma}). 
Let $\DD$ be the set of codimension one points of $\XX_0$ consisting of 
support of $d$, support of $\lambda$, the closed fibre $X_0$ and 
the ramification locus of $D$ on $\XX_0$.   Let $P \in \XX_0$ be a
closed point. Then, by the choice of  $\XX_0$, there exist at most two codimension one 
points of $\XX_0$ which are in $\DD$ and pass through $P$.
Further, since  
$\XX$ is regular,  there exists at most one codimension one point $\eta$ of $\XX_0$ 
passing through $P$ such that 
$\nu_\eta(d)$ is odd.

Let $P \in \XX_0$ be a closed point. Let  $\hat{R}_{0P}$ be the  completion of the 
local  ring at 
$P$ on $\XX_0$,  $m_P$ the maximal ideal $\hat{R}_{0P}$, $F_{0P}$ the field of fractions of 
$\hat{R}_{0P}$ and $F_P = F_{0P} \otimes F$. 
Let $w_P \in \hat{R}_{0P}$ be a unit which is not 
a square in $\hat{R}_{0P}$. 
Since the residue field $\kappa(P)$ at $P$ is a finite field, any unit in $\hat{R}_{0P}$ is a square or 
$w_P$ times a square.

Let $\PP_0$ be the  finite set of closed points of $\XX_0$  consisting of  the  points of
intersection of   two distinct  codimension  one   points  in   $\DD$.  

Let $P \in \PP_0$ and  $\eta_1, \eta_2 \in \DD$  such that 
$P \in  \overline{\{ \eta_1 \}} \cap \overline{\{ \eta_2 \}}$. 
Then  $m_P = (\pi_P, \delta_P)$ with  $\eta_1$ and $\eta_2$ are given by  primes 
$\pi_P$ and $\delta_P$ respectively at $P$, 
$d =  d_P^2$ or $d = w_Pd_P^2$ or $d = w_P\pi_Pd_P^2$  or $d = w_P\delta_Pd_P^2$  and 
$\lambda = u_P\pi_P^{r} \delta_P^{s}$
for some $u_P  \in \hat{R}_{0P}$ units, $d_P \in F_{0P}^*$, 
$r = \nu_{\eta_1}(\lambda)$, $s  =  \nu_{\eta_2}(\lambda)$  and $D$ is unramified 
at $P$ except possibly at $(\pi_P)$ and $(\delta_P)$. 
Let $\epsilon_1, \epsilon_2 \in \{ 0, 1\}$ and $r_1, s_1 \in {\mathbb Z }$ be such that
$r = 2r_1 + \epsilon_1$ and $s = 2s_1 + \epsilon_2$. Then 
$\lambda   = u_P\pi_P^{\epsilon_1} \delta_P^{\epsilon_2} (\pi_P^{r_1}\delta_P^{s_1})^2$.

 Suppose that the period of $D$ is 2.  Then ind$(D) \leq 4$  (cf. \cite{salt}) and 
ind$(D\otimes F_{0P}) \leq 2$ (cf. \ref{2dim-quat} ) for all closed points $P$ of $\XX_0$.
Since $D$ has a $F/F_0$-involution, for every  closed point $P \in \XX_0$,  
 there exists a central  division  algebra $D_{0P}$ over $F_{0P}$ such that 
$D\otimes F_{0P} = D_{0P}\otimes F_P$ (cf. \cite[Proposition 2.22]{knus}) and $D_{0P}$ is unramified at $P$ except possibly at 
$\eta_1$ and $\eta_2$ (cf. \ref{2dim-class} ). Further if  $D\otimes F_{0P}$ is a split algebra,  we choose
$D_{0P} = F_{0P}$ and 
if $D\otimes F_{0P}$ is not a split algebra,   $D_{0P}$ be as in (\ref{2dim-class}).

\begin{prop} 
\label{closed-d-unit}  Suppose $\nu_{\eta_1}(d)$ and $\nu_{\eta_2}(d)$ are even.  
Then there exist $a_{iP}$, $\mu_{iP}$, $i = 1, 2,  3$  such that 

i) $a_{1P} =    v_P \pi_P^{\epsilon_1}\delta_P^{\epsilon_2} \in  F_{0P} \setminus F_{0P}^{*2}$, 
 $v_P$ a unit at $P$, 
 $\mu_{1P} \in L_{1P} =  F_{0P}[X]/(X^2 - a_{1P})$
 
ii)  $a_{2P} \in \hat{R}_{0P}$  a unit at $\eta_1$ and $\eta_2$ and 
$\partial_{\eta_i}(D_{0P}) = \bar{a}_{2P} \in \kappa(\eta_i)^*/\kappa(\eta_i)^{*2} $ for $i = 1, 2$

iii) $a_{3P} = d a_{2P}$,
 
 iv) $\mu_{iP} \in F_{0P}[X]/(X^2 - a_{iP})^*$  unit along $\pi$ and $\delta$ for $i = 2, 3$
 
 v) $\prod_i N_{L_{iP}/F_{0P}}(\mu_{iP}) = \lambda$, where $L_{iP} = F_0[X]/(X^2 - a_i)$ for $i = 1, 2, 3$ 
 
 vi) $\mu_{iP}$ is a reduced norm from $D \otimes L_{iP}$ for $i  =1, 2 , 3$
 \end{prop}

\begin{proof}  Suppose $D\otimes F_{0P}$ is a split algebra.
Let $v_P$ a unit at $P$ such that $a_{1P} = v_P 
\pi_P^{\epsilon_1}\delta_P^{\epsilon_2} \in  F_{0P} \setminus F_{0P}^{*2}$. 
Then    
  $\mu_{1P} = \pi_P^{r_1}\delta_P^{s_1} \sqrt{a_{1P}}$, 
   $a_{2P} = 1$,  $\mu_{2P}  = (-v_P^{-1}u_P, 1) 
   = F_{0P} \times F_{OP} =  F_{0P}[X]/(X^2 - 1)$ and 
   $  \mu_{3P} = 1$   have the required properties.

Suppose that $D\otimes F_{0P}$ is not a split algebra. 
Suppose $d$ is not a square in $F_{0P}$.  Since 
$\nu_{\eta_1}(d)$ and $\nu_{\eta_2}(d)$ are even,  
$d = w_Pd_1^2$ for some $d_1 \in F_{0P}^*$. Then, by (\ref{2dim-class}), 
$D_{0P} = (\pi_P, \delta_P)$.    Then $a_{1P}, a_{2P}, \mu_{iP}$ as in (\ref{2dim-unit})
have the required properties. 

Suppose $d$ is a square in $F_{0P}$. Then $F \otimes F_{0P} = F_{0P} \times F_{0P}$
and $D\otimes F_{0P} = D_{0P} \times D_{0P}^{op}$ for some quaternion algebra $D_{0P}$ over 
$F_{0P}$.  Further $D_{0P}$ is unramified at $P$ except possibly at $(\pi_P)$ and $(\delta_P)$. 
Let $a_{1P}$, $a_{2P}$ and $\mu_i \in L_{iP}$ be as in (\ref{2dim-square}). Let  $\mu_{3P} = 1$. 
Then $a_{1P}$, $a_{2P}$ and $\mu_{iP}$, $i = 1, 2, 3$  have the required properties.  
\end{proof}

\begin{prop}  
\label{closed-d=pi}
Suppose $\nu_{\eta_1}(d)$ is odd. 
Then  there exist $a_{1P}$, $a_{2P}$, $\mu_{1P}$, $\mu_{2P}$  
such that 

i) $a_{1P} =    v_P \pi_P^{\epsilon_1}\delta_P^{\epsilon_2} 
\in  F_{0P} \setminus F_{0P}^{*2}$,  $v_P$ a unit at $P$, 
 $\mu_{1P} \in L_{1P} =  F_{0P}[X]/(X^2 - a_{1P})$
 
ii) $a_{2P} \in \hat{R}_{0P}$  a unit at $P$  and 
$\partial_{\eta_2}(D_{0P}) = \bar{a}_{2P} \in \kappa(\eta_2)^*/\kappa(\eta_2)^{*2} $

 iii) $\mu_{2P} \in \hat{R}_{0P}[X]/(X^2 - a_{2P})^*$ 
 
 iv) $  N_{L_{1P}/F_{0P}}(\mu_{1P}) N_{L_{2P}/F_{0P}}(\mu_{2P}) = \lambda$, 
 where $L_{iP} = F_0[X]/(X^2 - a_i)$ for $i = 1, 2$ 
 
 vi) $\mu_{iP}$ is a reduced norm from $D \otimes L_{iP}$ for $i  =1, 2$
\end{prop}

\begin{proof} 
Suppose $D\otimes F_{0P}$ is a split algebra.
Let $v_P$ be a unit at $P$ such that $a_{1P} = v_P 
\pi_P^{\epsilon_1}\delta_P^{\epsilon_2} \in  F_{0P} \setminus F_{0P}^{*2}$. 
Then    
  $\mu_{1P} = \pi_P^{r_1}\delta_P^{s_1} \sqrt{a_{1P}}$, 
   $a_{2P} = 1$ and  $\mu_{2P}  = (-v_P^{-1}u_P, 1) 
   = F_{0P} \times F_{OP} =  F_{0P}[X]/(X^2 - 1)$  have the required properties.

Suppose $D\otimes F_{0P}$ is not a split algebra. 
Since $\nu_{\eta_1}(d)$ is odd, by the choice of  $\XX_0$, 
$\nu_{\eta_2}(d)$ is even. Hence $d= v_P \pi_P d_1^2$ for some 
$v_1\in \hat{R}_{0P}$ a unit and $d_1 \in F_{0P}^*$. 
In particular $D_{0P} = (w_P, \delta_P)$. 
Let $a_{2P} = w_P$. 
Hence, by  (\ref{choice-at-P-pi}),  there exist 
$a_{1P} = v_P\pi^{\epsilon_1} \delta^{\epsilon_2} \in 
F_{0P}^* \setminus F_{0P}^{*2}$  with $v_P\in \hat{R}_{0P}^*$   
and $\mu_{1P} \in L_{1P} = F_{0P}[X]/(x^2 - a_{1P})$ and 
$\mu_{2P} \in L_{2P} = F_{0P}[X]/(x^2 - a_{2P})$  such that 

i)  $  N_{L_{1P}/F_{0P}}(\mu_{1P})  N_{L_{2P}/F_{0P}}(\mu_{2P})= \lambda$

ii)  $\mu_{iP}$ is a reduced from from $ D \otimes  L_{iP}$ for $i = 1, 2$

iii)  $\mu_{2P} \in \hat{R}_{0P}[X]/(X^2 - a_{2P})^*$.

Then  $a_{1P}$, $a_{2P}$, $\mu_{1P}$ and $\mu_{2P}$  have the required properties.  
\end{proof}

\section{Choices at codimension one points and curve points}\label{section7}

Let  $p \geq  3$ be  a prime and  $K$ be a  $p$-adic field.
Let $F_0$ be the function field of a curve over $K$ and 
$F = F_0(\sqrt{d})$ a quadratic field extension.  
Let $D$ be a central division algebra over $F$ with a $F/F_0$-involution. 
Suppose that period of $D$ is 2.  
Let $\lambda \in F_0^* \cap Nrd(D)^*$.

 Let $\XX_0$, $\XX$, $\DD$ and $\PP_0$ be as in (\ref{nodalpoints}). 
 Let $\eta \in X_0$ be a codimension zero  point.  
 Let $\pi$ be a parameter at $\eta$. 
 Then, by (\ref{dvr-class-4}) and (\ref{dvr-class-2}), 
we have  $D\otimes F_{0\eta} = (b, c) \otimes (w, \pi)$ 
for some  $b, c, w \in F_{0\eta}$.  Let $D_{0\eta} =(b, c) \otimes (w, \pi)$.
 Write $\lambda = u \pi^r$
 for some $u \in F_{0\eta}$ with  $\nu_\eta(u) = 0$.  
Let $\epsilon   \in \{ 0, 1\}$ and $r_1 \in {\mathbb Z }$ such that 
$r = 2r_1 + \epsilon$.  Then $\lambda   = u \pi ^{\epsilon}  (\pi ^{r_1})^2$. 
Let $\PP_\eta = \PP_0 \cap\overline{\{ \eta \}}$.  

\begin{prop}
\label{codim1-even-even}
 Let $\eta \in X_0$ be a codimension  zero point.  Suppose 
 $\nu_\eta(d)$ and  $\nu_\eta(\lambda)$ are  even. 
For each $P \in \PP_\eta$, if $d$ is a unit at $P$ up to a square $F_{0P}$,  
 let  $a_{iP}$ and $\mu_{iP}$, $i = 1, 2, 3$ be as in (\ref{closed-d-unit})  and 
if $d$ is not a unit at $P$ up to a square in $F_{0P}$, let 
$a_{iP}$  and $\mu_{iP}$, $i = 1, 2$ be as in (\ref{closed-d=pi}),
$a_{3P} = da_{2P}$ and $\mu_{3P} = 1$.
Then there exist $a_{1\eta} , a_{2\eta}, a_{3\eta} \in F_{0\eta}$  units at $\eta$ and 
$\mu_{i\eta} \in F_{0\eta}[X]/(X^2 - a_{i\eta})$ such that 

i) $\prod_i N_{L_{i\eta}/F_{0\eta}}(\mu_{i\eta}) = \lambda$

ii) $\mu_{i\eta}$ is a reduced norm from $D\otimes L_{i\eta}$  for $i = 1, 2,  3$

iii) ind$(D\otimes L_{i\eta}) \leq {2}$ for all  $P \in \PP_\eta$ and   $i = 1, 2,  3$

iv) $a_{i\eta}a_{iP} \in F_{0P, \eta}^{*2}$ for all  $P \in \PP_\eta$ and   $i = 1, 2,  3$

v)  for $P \in \PP_\eta$, there is an isomorphism 
$$\phi_{iP, \eta} : F_{0P,\eta}[X]/(X^2 - a_{i\eta}) \to  F_{0P,\eta}[X]/(X^2 - a_{iP})$$

 such that 

$$\phi_{iP, \eta}(\mu_{i\eta} ) \mu_{iP}^{-1}   \in   (F_{0P, \eta}[X]/(X^2 - a_{iP}) )^{2^m}$$ for all $m \geq 1$ and  $i = 1, 2,  3$. 
\end{prop}

\begin{proof}  Since $\nu_\eta(d)$ is even,   
replacing $d$ by $d$ times a square in $F_{0\eta}$,
we assume that $\nu_\eta(d) = 0$.

Let $\pi_\eta$ be a parameter at $\eta$ such that 
for every $P \in \PP_\eta$, the maximal ideal at $P$ is given by 
$(\pi_\eta, \delta_P)$ for some prime $\delta_P$.

By (\ref{dvr-class-4}) and (\ref{dvr-class-2}), we have 
$D\otimes F_\eta = (b, c) \otimes (w, \pi_\eta)$ for some $b, c, w \in F_{0\eta}$ which are units at $\eta$. 
Denote $D\otimes F_{\eta}$ by $D_{0\eta}.$ Let
$u_0$, $b_0$, $c_0$, $d_0$ and $w_0$ be the images of $u, b, c,  d, w$ in $\kappa(\eta)$ respectively. 
Since $\lambda = u\pi_{\eta}^r$ with $u\in F_{0\eta}$ a unit at $\eta$
and $\lambda$ is a reduced norm from $D\otimes F_\eta$, the cup product $\partial(u\pi_{\eta}^r)\cup [(b,c)\otimes (w, \pi_{\eta})]=0 \in H^3(F_{\eta}, \mu_2^{\otimes2}).$ Since $r$ is even, $\partial(u\pi_{\eta}^r)\cup [(b,c)]=0.$ Then $ \partial(u\pi_{\eta}^r)\cup [(w,\pi_{\eta})]=(w,u)=0 $ and we have $u \in N_{F_{0\eta}(\sqrt{d}, \sqrt{w})/F_{0\eta}(\sqrt{d})}(F_{0\eta}(\sqrt{d}, \sqrt{w})^*)$ (cf. \cite[Theorem 8.9.1]{GS2017}, \cite[section 4]{PPS}). Hence $u_0 \in  N_{\kappa(\eta)(\sqrt{d_0}, \sqrt{w_0})/
F_0(\sqrt{d_0})}(\kappa(\eta)(\sqrt{d_0}, \sqrt{w_0})^*)$.

Since $\nu_\eta(\lambda) = 2r_1$, 
by the choice of  $a_{1P}$  (\ref{closed-d-unit}, \ref{closed-d=pi}), we have 
$a_{1P} = v_P  \delta_P^{\epsilon_2}$ for some unit $v_P$ at $P$
and $\epsilon_2 \in \{ 0, 1\}$. Further $a_{1P}$ is not  square in $F_{0P}$. Since $F_{0P,\eta}$ is the completion of $F_{0P}$
at $\eta$, $a_{1P}$ is not a square in $F_{0P, \eta}$.
Let $x_P = \bar{a}_{1P} =  \bar{v}_P\bar{\delta}_P^{\epsilon_2}$. 

  Since $\mu_{iP} \in R_{0P}[X]/(X^2 - a_{iP})$
are units along $\eta$ for $i = 2, 3$ and $\prod_1^3 N_{L_{iP}/F_{0P}}(\mu_{iP}) = \lambda$, it follows that 
$\nu_\eta(N_{L_{1P}/F_{0P}}(\mu_{1P}))  = \nu_\eta(\lambda)$. 
Since $L_{1P} \otimes F_{0P, \eta} = F_{0P, \eta}[X]/(X^2 - a_{1P})$ is unramified
and $\nu_\eta(N_{L_{1P}/F_{0P}}(\mu_{1P}))  = \nu_\eta(\lambda) = 2r_1$,
we have  $\mu_{1P} =  y'_P \pi_\eta^{r_1}$ for some $y'_P \in L_{1P} \otimes F_{0P, \eta}$ unit 
in the valuation ring.  Let $y_{1P}  = \bar{y'}_P \in \kappa(\eta)_{1P} = \kappa(\eta)_P[X]/(X^2 - x_P)$. 

For $i = 2, 3$,  let $y_{iP}$ be the image of $\mu_{iP}$ in $\kappa(\eta)_P[X]/(X^2 - \bar{a}_{2P}).$
  By the choice (\ref{closed-d-unit},  \ref{closed-d=pi}), we have 
 $a_{2P} = \partial_{\eta}(D_{0\eta}) = \bar{w}$ and $a_{3P} = da_{2P}$. 
 Then  $y_{2P} = \bar{\mu}_{2P} \in \kappa(\eta)_{2P} = 
 \kappa(\eta)_P[X]/(X^2 - \bar{a}_{2P}) = \kappa(\eta)_P[X]/(X^2 -  \bar{w})$ and 
  $y_{3P} = \bar{\mu}_{3P} \in \kappa(\eta)_{3P} =\kappa(\eta)_P[X]/(X^2 - \bar{a}_{3P}) = 
 \kappa_\nu[X]/(X^2 -  \bar{d}\bar{w})$. 

Further we  have  

i) $\prod_i N_{\kappa(\eta)_{iP}/\kappa(\eta)_P}(y_{iP}) = \bar{u}$. 

ii) $y_{1P} \in N_{\kappa(\eta)_P(\sqrt{\bar{d}}, \sqrt{\bar{w}})/\kappa(\eta)_P}
(\kappa(\eta)_P(\sqrt{\bar{d}}, \sqrt{\bar{w}})^*)$. 

Hence, by (\ref{dvr-unit}), there exists $a \in \hat{R}_{0\eta}^*$,
$\mu_{1 } \in \hat{R}_{0\eta}[X]/(X^2 - a)$, $\mu_{2}
 \in \hat{R}_{0\eta}[X]/(X^2 - w) $ and 
$\mu_{3} \in \hat{R}_{0\eta}[X]/(X^2 - d w) $  such that

i) $\bar{a} $ is close to $x_P$ and $\bar{\mu}_{i}$ is close to $y_{iP}$ 
for all $P  \in \PP_0$  and $i = 1, 2, 3$ 
 
 ii) $\prod_i N_{L_{i\eta}/F_{0\eta}}(\mu_{i}) =  u$, where $L_{1\eta} =   F_{0\eta}[X]/(X^2 - a)$, 
  $L_{2\eta} = F_{0\eta}[X]/(X^2 - w)$ and  $L_{3\eta} = F_{0\eta}[X]/(X^2 - dw)$
  
  iii) $(b, c) \otimes L_{1 \eta}$ is split
 
 iv) $\mu_i$ is a reduced norm from $D \otimes_{F_{0\eta}} L_{i\eta}$ for $i = 1, 2, 3$.  

 Since $D  \otimes L_{1\eta} =  ( (b, c) \otimes L_{1\eta} ) \otimes (w, \pi_\eta) \otimes L_{1\eta})
 =  (w, \pi_\eta) \otimes L_{1\eta})$, $\pi_\eta$ is a reduced norm from $D  \otimes L_{1\eta}$.
 Hence $a_{1\eta} = a $, $a_{2\eta} = w$, $a_{3\eta} = wd$, $\mu_{1\eta}  =  \mu_1 \pi_\eta^{r_1}$,
 $\mu_{2\eta} = \mu_2$ and $\mu_{3\eta} = \mu_3$ have the required properties. 
  \end{proof}
 
 \begin{prop}
\label{codim1-even-odd}
 Let $\eta \in X_0$ be a codimension  zero point.  Suppose 
 $\nu_\eta(d)$ is even  and  $\nu_\eta(\lambda)$ is odd. 
For each $P \in \PP_\eta$, if $d$ is a unit at $P$ up to a square $F_{0P}$,  
 let  $a_{iP}$ and $\mu_{iP}$, $i = 1, 2, 3$ be as in (\ref{closed-d-unit})  and 
if $d$ is not a unit at $P$ up to a square in $F_{0P}$, let 
$a_{iP}$  and $\mu_{iP}$, $i = 1, 2$ be as in (\ref{closed-d=pi}),
$a_{3P} = da_{2P}$ and $\mu_{3P} = 1$.
Then there exist $a_{1\eta} , a_{2\eta}, a_{3\eta} \in F_{0\eta}$ and 
$\mu_{i\eta} \in F_{0\eta}[X]/(X^2 - a_{i\eta})$ such that 

i) $\prod_i N_{L_{i\eta}/F_{0\eta}}(\mu_{i\eta}) = \lambda$

ii) $\mu_{i\eta}$ is a reduced norm from $D\otimes L_{i\eta}$  for $i = 1, 2,  3$

iii) ind$(D\otimes F_{i\eta}) \leq 2$ for $i = 1, 2,  3$

iv) $a_{i\eta}a_{iP} \in F_{0P, \eta}^{*2}$ for $i = 1, 2,  3$

v) for $P \in \PP_\eta$, there is an isomorphism 
$$\phi_{iP, \eta} : F_{0P,\eta}[X]/(X^2 - a_{i\eta}) \to  F_{0P,\eta}[X]/(X^2 - a_{iP})$$

 such that 

$$\phi_{iP, \eta}(\mu_{i\eta} ) \mu_{iP}^{-1}   \in  
  (F_{0P, \eta}[X]/(X^2 - a_{iP}) )^{2^m}$$ for all $m \geq 1$ and  $i = 1, 2,  3$. 
\end{prop}

\begin{proof}   
By \cite[Proposition 5.8 and Corollary 5.9]{PPS}, $\text{ind}(D)=\text{ind}(D \otimes F_{0\eta})=\text{per}(D)=2$ is 2.

Since $\nu_\eta(\lambda)$ is odd and $\partial(\lambda)\cup [D \otimes F_{0\eta}]=0$,
  $D\otimes F_{0\eta} = (w, \pi_\eta)$  
   for some   parameter $\pi_\eta$  at $\eta$ and $w  \in F_0^*$   a  unit  at $\eta$. 

Since $\nu(\lambda) $ is odd, $\pm \lambda$ is not a square in $F_{0P}$ for all
$P \in \PP_\eta$. Hence, by the choice of $a_{1P}$ and $\mu_{iP}$,
we have $a_{1P}\lambda \in F_{0P, \eta}^{*2}$, $\mu_{1P} \sqrt{\lambda}
\in F_{P\eta}(\sqrt{a_{1P}})^{*2}$. Further 
$w a_{2P} \in F_{0P,\eta}^{*2}$, $a_{3P} = a_{2P}d$, $\mu_{1P} = \mu_{2P} = 1$. 

Let $a_{1\eta} = -\lambda$, $a_{2\eta} = w$,  $a_{3\eta} = dw$, $\mu_{1\eta} = \sqrt{-\lambda}$
and $\mu_{2\eta} = \mu_{2\eta} = 1$.
Since $\lambda$ is a reduced norm from $D$, we have $(\lambda) \cdot D = 0 \in H^3(F, \mu_2)$.
Since $\nu(\lambda)$ is odd,   if $D\otimes F_\eta$ is not split, then 
 by (\cite[Lemma 4.7]{PPS}),  ind$(D \otimes F_\eta(\sqrt{a_{1\eta}}))  < $ ind$(D\otimes F_\eta)$.
In particular  $D\otimes F_{\eta}(\sqrt{a_{1\eta}})$ is split.  
Hence $a_{1\eta} = -\lambda$, $a_{2\eta} = w$, $a_{3\eta} = d w$, $\mu_{1\eta} = \sqrt{-a_{1\eta}}$
and $\mu_{2\eta} = \mu_{3\eta} = 1$ have the required properties.   \end{proof}

 \begin{prop}
\label{codim1-d=pi}
 Let $\eta \in X_0$ be a codimension zero point.  Suppose 
 $\nu_\eta(d)$ is odd. 
For each $P \in \PP_\eta$,  let 
$a_{iP}$  and $\mu_{iP}$, $i = 1, 2$ be as in (\ref{closed-d=pi}). 
Then there exist $a_{1\eta} , a_{2\eta} \in F_{0\eta}$ and 
$\mu_{i\eta} \in F_{0\eta}[X]/(X^2 - a_{i\eta})$ such that 

i) $\prod_i N_{L_{i\eta}/F_{0\eta}}(\mu_{i\eta}) = \lambda$

ii) $\mu_{i\eta}$ is a reduced norm from $D\otimes L_{i\eta}$  for $i = 1, 2$

iii) ind$(D\otimes F_{i\eta}) \leq 2$ for $i = 1, 2$

iv) $a_{i\eta}a_{iP} \in F_{0P, \eta}^{*2}$ for $i = 1, 2$

v) for $P \in \PP_\eta$, there is an isomorphism 
$$\phi_{iP, \eta} : F_{0P,\eta}[X]/(X^2 - a_{i\eta}) \to  F_{0P,\eta}[X]/(X^2 - a_{iP})$$
 such that 
$$\phi_{iP, \eta}(\mu_{i\eta} ) \mu_{iP}^{-1}   \in    (F_{0P, \eta}[X]/(X^2 - a_{iP}) )^{2^m}$$ 
for all $m \geq 1$ and  $i = 1, 2$. \end{prop}

\begin{proof}  
By (\ref{dvr-class-4}) and (\ref{dvr-class-2}), we have 
$D\otimes F_\eta = (b, c)$ for some $b, c \in F_{0\eta}$ which are units at $\eta$. 
Let $D_{0\eta} = (b, c) $ and 
$u_0$, $b_0$, $c_0$  be the images of  $u, b, c $ in $\kappa(\eta)$. 

Write $r = 2r_1 + \epsilon_1$ for some $r_1 \in {\mathbb Z}$ and 
$\epsilon_1 \in \{ 0, 1\}$.  By the choice of $a_{1P}$, we have 
$a_{1P} = w_P^{\epsilon_P} \pi_{\eta}^{\epsilon_1}$
for some $w_P \in F_{0P, \eta}$ unit  at $\eta$.  
 Let $x \in \kappa(\eta)$ be close to 
$\bar{w}_P $ for all $P \in \PP_\eta$. 
Let $a \in F_{0\eta}$ which maps to $x$ in $\kappa(\eta)$ and $a_{1\eta} = 
a \pi_\eta^{\epsilon_1}$. Then $a_{1\eta} a_{1P} \in F_{0P, \eta}^{*2}$ for all 
$P \in \PP_\eta$.   Let $L_{1\eta} = F_{0\eta}(\sqrt{a_{1\eta}})$ and 
$L_{1P, \eta}  = F_{0P, \eta}(\sqrt{a_{1P}})$. Then $ L_{1\eta} \otimes F_{0P,\eta} = L_{1P, \eta}$. 
 Let   $\kappa(\eta)_1$  be the 
residue field of $L_{1\eta}$. Then $\kappa(\eta)_{1P}$ is the residue field of 
$F_{0P, \eta}(\sqrt{a_{1P}})$.

 Since $\mu_{2P} \in R_{0P}[X]/(X^2 - a_{2P})$
is a  unit  along $\eta$ and $ N_{L_{1P}/F_{0P}}(\mu_{1P}) N_{L_{2P}/F_{0P}}(\mu_{2P}) = \lambda$, 
it follows that  $\nu_\eta(N_{L_{1P}/F_{0P}}(\mu_{1P}))  = \nu_\eta(\lambda) = 2r_1 + \epsilon_1$. 

Suppose  $\epsilon_1 = 0$. Then $L_{1\eta}/F_{0, \eta}$ is unramified and $\pi_\eta$ is 
a parameter in $F_{0P, \eta}(\sqrt{a_{1P}}) $. Hence $\mu_{1P} = \theta_P \pi_{\eta}^{r_1}$ for some 
$\theta_P \in F_{0P, \eta}(\sqrt{a_{1P}})$ a unit at $\eta$.
 Suppose $ \epsilon_1 = 1$.  Then $L_{1\eta}/F_{0\eta}$ is 
ramified and $\sqrt{a_{1\eta}}$ is a parameter.  
Hence $\mu_{1P}  =  \theta_P\pi_\eta^{r_1} \sqrt{a_{1\eta}}$ for some 
$\theta_P \in F_{0P, \eta}(\sqrt{a_{1P}})$ a unit at $\eta$.
In both cases, let $\theta \in \kappa(\eta)_1$ close to $\bar{\theta}_P$ for all $P \in \PP_\eta$
and $\theta_1 \in L_{1\eta}$ which lifts $\theta$. 
If $\epsilon_1 = 0$, let $\mu_{1\eta} = \theta_1 \pi_\eta^{r_1}$ and 
if $\epsilon_1 = 1$, let $\mu_{1\eta} = \theta_1\pi_\eta^{r_1} \sqrt{a_{1\eta}}$. 
Then $\mu_{1\eta} \mu_{1P} \in L_{1\eta, P}^{*2}$ and 
 $\lambda N_{L_{1\eta}/F_{0\eta}}(\mu_{1\eta})^{-1}$ is a unit at $\eta$. 
 Since  $\nu(d)$ is odd, $F/F_0$ is ramified at $\eta$ and hence 
 $\mu_{1\eta} = \mu_{1\eta}' g_\eta^2$ for some $\mu_{1\eta}' \in  F\otimes L_{1\eta}$ 
 a unit at $\eta$ and $g_\eta \in  F\otimes L_{1\eta}$.  
 Since  $\kappa(\eta)_1$ is a global field of positive characteristic, 
 $\bar{\mu}_{1\eta}' \in \kappa(\eta)_1$ is a reduced norm from $ (b_0, c_0) \otimes \kappa(\eta)_1$ (cf. \cite[Theorem 33.5]{reiner}).
 Hence $\mu_{1\eta}$ is a reduced norm from $D \otimes L_{1\eta}$.

 Let $z_1 \in \kappa(\eta)$ be the image of  $\lambda N_{L_{1\eta}/F_{0\eta}}(\mu_{1\eta})^{-1}$ and
  $y_{2P}$ be the image of $\mu_{2P}$ in $ \kappa(\eta)_{2P} = \kappa(\eta)_P[X]/(X^2 - \bar{a}_{2P}).$
  By the choice of $\mu_{1\eta}$, it follows that 
  $N_{\kappa(\eta)_{2P}/\kappa(\eta)_P}(y_{2P})$ is close to $z_1$.
  Hence, replacing $y_{2P}$ by some element which is close to $y_{2P}$, we assume that 
  $N_{\kappa(\eta)_{2P}/\kappa(\eta)_P}(y_{2P}) = z_1$. In particular 
  the quaternion algebra $(\bar{a}_{2P}, z_1)$ is split over $\kappa(\eta)_P$ for all $P \in \PP_\eta$.
   Hence  $\bar{a}_{2P}$ is a norm from the extension $\kappa(\eta)_P[X]/(X^2 - z_1)$.
   Let $\tilde{a}_{2P} \in \kappa(\eta)_P[X]/(X^2 - z_1)$ with norm equal to $\bar{a}_{2P}$. 
   Let $\tilde{a}_2 \in \kappa(\eta)[X]/(X^2 - z_1)$ be   close to $\tilde{a}_{2P}$ for all $P \in \PP_\eta$
    and $\bar{a}_2$ be the norm of $\tilde{a}_2$. Then $\bar{a}_2$ is close to $\bar{a}_{2P}$ for all 
    $P \in \PP_\eta$.  Since the quaternion algebra $(\bar{a}_2, z_1)$ is split, $z_1$ is a norm from the 
    extension $\kappa(\eta)_2 = \kappa(\eta)[X]/(X^2 - \bar{a}_2)$. 
  
  Then, there exists $y_2 \in \kappa(\eta)_2$
  which is close to $y_{2P}$ for all $P \in \PP_\eta$ such that 
  $N_{\kappa(\eta)_2/\kappa(\eta)}(y_{2}) = z_1$ since $\kappa(\eta)_2$ is a global field. 
  Let $a_{2\eta} \in F_{0\eta}$ be a  lift of $\bar{a}_2 \in \kappa(\eta)$ and 
  $\mu_{2\eta} \in L_{2\eta} = F_{0\eta}[X]/(X^2 - a_{2\eta})$ be such that
  $N_{L_{2\eta}/F_{0\eta}}(\mu_{2\eta})= \lambda N_{L_{1\eta}/F_{0\eta}}(\mu_{1\eta})^{-1}$. 
  Since $\mu_{2\eta}$ is a unit at $\eta$ and $D$ is unramified at 
  $\eta$, as above, $\mu_{2\eta}$ is a reduced norm from $D\otimes L_{2\eta}$. 

   Hence $a_{1\eta} $, $a_{2\eta} $, $\mu_{1\eta} $ and $\mu_{2\eta} $  have the required properties. 
  \end{proof}
  
  \begin{prop} 
\label{curve-point-choice} Let $P \in \XX_0$ be a closed point. Suppose that 
$P \not\in \PP_0$.  Let $\eta \in D$ be the unique codimension one point with $P \in \overline{ \{ \eta \} }$.
Let $ a_{i\eta} \in F_{0\eta}$ and 
$\mu_{i\eta} \in L_{i\eta} = F_{0\eta}[X]/(X^2 - a_{i\eta})$   be as in 
(\ref{codim1-even-even},   \ref{codim1-even-odd}, \ref{codim1-d=pi}). 
Then there exist $a_{iP} \in F_{0P}$ and $\mu_{iP} \in L_{iP} = F_{0P}[X]/(X^2 - a_{iP})$ 
such that 

i) $\prod_i N_{L_{iP}/F_{0P}}(\mu_{iP}) = \lambda$

ii) $\mu_{iP}$ is a reduced norm from $D\otimes L_{iP}$  

iii) $a_{i\eta}a_{iP} \in F_{0P, \eta}^{*2}$  

iv) there is an isomorphism 
$$\phi_{iP, \eta} : F_{0P,\eta}[X]/(X^2 - a_{i\eta}) \to  F_{0P,\eta}[X]/(X^2 - a_{iP})$$
 such that 
$$\phi_{iP, \eta}(\mu_{i\eta} ) \mu_{iP}^{-1}   \in    (F_{0P, \eta}[X]/(X^2 - a_{iP}) )^{2^m}$$ 
for all $m \geq 1$.  
\end{prop}

\begin{proof}  Let $\pi_P$ be a prime  defining $\eta$ at $P$. 
  Since there is a unique codimension one point in $\DD$, the support of $d$ at $P$
  and the ramification locus of at $P$  is at most $\eta$.
  Hence, by (\ref{curve-point-choice0}), we have the required  $a_{iP}$ and $\mu_{iP}$.
\end{proof}

\section{Choice of $U$}\label{section8}

 Let    $T$ be a   complete  discrete valuation ring with  field of fractions $K$ and 
  residue field  $k$. 
Let $F_0$ be the function field of a curve over $K$ and 
$F = F_0(\sqrt{d})$ a quadratic  etale extension.  
Let $D$ be a central division algebra over $F$ such that the period of $D$ is coprime to $\operatorname{char}(k)$.

 \begin{prop}\label{pro8.1} 
 \label{atu} Let $\XX_0$ be a normal proper model of 
 $F_0$ over $T$ and  $X_0$   the closed fibre of $\XX_0$.
 Let $\eta \in X_0$ be a codimension zero  point.   
 Let $\lambda \in F_0^* \cap Nrd(D)^*$,  $m \geq 2$ and $M \geq 1$. 
 Suppose  that for $1 \leq i \leq m$,  there exist  $a_{i\eta} \in F_{0\eta}$,
 $\mu_{i\eta} \in  L_{i\eta} = F_{0\eta}[X]/(X^2 -  a_{i\eta})$ such that 
 
 i) $\prod_1^m N_{L_{i\eta}/F_{0\eta}}(\mu_{i\eta}) = \lambda$
 
 ii) $\mu_{i\eta}$ is a reduced norm from $D\otimes L_{i\eta}$ for all $i$
 
 iii)  ind$(D\otimes L_{i\eta}) <  M$ for all $i$.
 
 Then there exist a non-empty open  proper subset  $U$ of $\overline{\{ \eta \}}$ and
 $a_{iU} \in F_{0U}$,
 $\mu_{iU} \in  L_{iU} = F_{0U}[X]/(X^2 - a_{iU}) $ such that 
 
 i) $a_{iU}a_{i\eta} \in F_{0\eta}^{*2}$
 
 ii)  there is an isomorphism  $$\phi_{i, U_\eta} :
  F_{0,\eta}[X]/(X^2 - a_{iU}) \to  F_{0,\eta}[X]/(X^2 - a_{i\eta})$$

 such that 

$$\phi_{i, \eta}(\mu_{iU} ) \mu_{i\eta}^{-1}   \in  (  F_{0P, \eta} [X]/(X^2 - a_{i\eta}))^{*2^m}$$
 for all $m \geq 1$ and  $i = 1, 2,  3$. 
 
 iii) $\prod_1^m N_{L_{iU}/F_{0U}}(\mu_{iU}) = \lambda$ for all $i$
 
 ii) $\mu_{iU}$ is a reduced norm from $D\otimes L_{iU}$ for all $i$
 
 iii)  ind$(D\otimes L_{iU}) <  M$ for all $i$. 

 \end{prop}
 
\begin{proof} Since $F_{0\eta}$ is the completion of 
$F_0$ and char$(k)\neq 2$,  there exists $a_i \in F_0^*$ such that 
$a_ia_{i\eta} \in F_{0\eta}^{*2}$. Thus, replacing $a_{i\eta}$ by 
$a_{i}$, we assume that $a_{i\eta}  = a_i \in F_0^*$.

Since $L_{i\eta} = F_{0\eta}(\sqrt{a_i})$ is the completion of 
$L_i = F_0(\sqrt{a_i})$, there exists  $\mu_i \in L_i $ close to 
$\mu_{i\eta}$ in $L_{i\eta}$.   In particular $\theta_i = 
N_{L_i/F_0}(\mu_i)^{-1}N_{L_{i\eta}/F_{0\eta}}(\mu_{i\eta})$
is close to 1 in $F_{0\eta}$.  Then $\theta = \prod_1^{m-1} \theta_i$ is close to 1 in $F_{0\eta}$. 
Let $\lambda_1 = \lambda  (\prod_1^{m-1} N_{L_{i }/F_{0 }}(\mu_{i }) ) ^{-1} \in F_{0}$. 
 Since 
 $\prod_1^m N_{L_{i\eta}/F_{0\eta}}(\mu_{i\eta}) = \lambda$, 
 we have  
 $$  
N_{L_m/F_0}(\mu_{m\eta})  =  \lambda_1  \theta^{-1}.$$

Since $\theta^{-1} \in F_{0\eta} $ is close to $1$, 
$\theta^{-1} = N_{L_{m\eta}/F_{0\eta}}(\theta')$ for some $\theta' \in L_{m\eta}$
which is close to 1.  In particular $\theta'$ is a reduced norm from $D\otimes L_{i\eta}$. 
 Hence replacing 
   $\mu_{m\eta}$ by $\mu_{m \eta}\theta'$, 
   we assume that 
    $$  
 N_{L_m/F_0}(\mu_{m\eta}) = \lambda_1.$$Hence, by (\cite[Lemma 7.2]{PPS}), there exists a nonempty proper open subset $U_0 $ of 
$\overline{  \{ \eta \} }$  and $\mu_{mU_0} \in L_m \otimes F_{0U_0}$ 
such that   $\lambda_1 = N_{L_m/F_0}(\mu_{mU_0})$ and $\mu_{mU_0}$ is close to 
$ \mu_{m\eta} $ in $L_m \otimes F_{0\eta}$.  

Since  ind$(D\otimes L_{i\eta}) <  M$ for all $i$, there exist 
nonempty proper open subsets $U_i$ of  $\overline{  \{ \eta \} }$ 
such that  ind$(D\otimes L_{i U}) <  M$ for all $i$. 

Then  $U  = (\cap_i U_i) \cap U$, $a_{iU } = a_i$ for all $i$, 
$\mu_{iU } = \mu_i$ for $1\leq i \leq m-1$ and $\mu_{mU } =  \mu_{mU_0}$
have the required properties. 
\end{proof}

\section{The main theorem}\label{section9}

\begin{theorem}\label{thm9.1}
Let  $p \geq  3$ be  a prime and  $K$ be a  $p$-adic field.
Let $F_0$ be the function field of a curve over $K$ and 
$F = F_0(\sqrt{d})$ a quadratic field extension.  
Let $D$ be a central division algebra over $F$ with a $F/F_0$-involution. 
Suppose that period of $D$ is 2.  
Let $\lambda \in F_0^* \cap Nrd(D)^*$.  Then there exist $a_i \in F_0^*$ and 
$\mu_i \in L_i = F_0[X]/(X^2 - a_i)$ for $i = 1, 2, 3$ such that 

i) $\prod_i N_{L_i/F_0}(\mu_i) = \lambda$

ii) $\mu_i$ is a reduced norm from $D\otimes L_i$ for $i = 1, 2, 3$

iii) ind$(D\otimes L_i) \leq 2$.
\end{theorem}

\begin{proof} Let $T$ be the  valuation ring of $K$ and $k$ the residue field of $K$.
Let $\XX_0$ be a regular proper model of $F_0$ over $T$ with  the union of the 
ramification locus of $D$, support of $d$, support of $\lambda$
 and the closed fibre $X_0$ of $\XX_0$   is a union of regular 
curves with  normal crossings. Further the integral closure $\XX$
of $\XX_0$ in $F$ is a  is a regular proper model of $F$.  

Let $\DD$ be the set of codimension one points of $\XX_0$ consisting of 
support of $d$, support of $\lambda$, the closed fibre $X_0$ and 
the ramification locus of $D$ on $\XX_0$.   Let $P \in \XX_0$ be a
closed point. Then, by the choice of  $\XX_0$, there exist at most two codimension one 
points of $\XX_0$ which are in $\DD$ and passes through $P$.
Further, since  
$\XX$ is regular,  there exists at most one codimension one point $\eta$ of $\XX_0$ 
passing through $P$ such that 
$\nu_\eta(d)$ is odd. 

Let $\PP_0$ be the  finite set of closed points of $\XX_0$  consisting of  points of the 
intersection of   the closures of  any two distinct  codimension  one   points  in   $\DD$.  

Let $P \in \PP_0$ and $\eta_1, \eta_2 \in \DD$ with $P \in \overline{ \{ \eta_1 \}} \cap 
\overline{ \{ \eta_2 \}}$.  If $\nu_1(d)$ and $\nu_2(d)$ are even, then let 
$a_{iP}, \mu_{iP}$ for $i = 1, 2, 3$ be as in (\ref{closed-d-unit}). If either 
$\nu_1(d)$ or $\nu_2(d)$ is odd, let $a_{iP}, \mu_{iP} $ for $i = 1, 2$ be as in (\ref{closed-d=pi})
and $a_{3P} = d a_{2P}$, $\mu_{3P} = 1$.

Let $\eta \in X_0 $ be a codimension zero point.
If $\nu(d)$ and $\nu(\lambda)$ are even, then let $a_{i\eta}, \mu_{i\eta}$ be 
as in (\ref{codim1-even-even}) for $i = 1, 2, 3$. If $\nu(d)$  is even and $\nu(\lambda)$ is odd,
 then let $a_{i\eta}, \mu_{i\eta}$ be  as in (\ref{codim1-even-odd}) for $i = 1, 2, 3$.
If $\nu(d)$  is odd, then let $a_{i\eta}, \mu_{i\eta}$ be  as in (\ref{codim1-d=pi}) for $i = 1, 2$
and $a_{3\eta} = a_{2\eta} d$, $\mu_{3\eta} = 1$.

Let $U_\eta$, $a_{iU_\eta}$ and $\mu_{iU_\eta}$ be as in (\ref{atu}). 
If necessary,     replacing each $U_\eta$ by a open subset of $U_\eta$, 
we assume that $\PP_0 \cap U_\eta = \emptyset$.  Let $\UU = \{ U_\eta \}$. 

Let $\PP = X_0 \setminus \cup_\eta U_\eta$. Then  $\PP_0$ is a finite set of 
closed points of  $\PP_0 \subseteq \PP$.

Let $P \in \PP \setminus \PP_0$. Then there is a unique codimension one point $\eta \in \DD$.
Let $a_{iP}$ and $\mu_{iP}$ for $i = 1, 2, 3$  be as in (\ref{curve-point-choice}).

 Let $P \in \PP$ and $U \in \UU$ with $P \in \overline{ \{\eta \}}$.
 Then, by the choice of $a_{iP}$ and $a_{iU}$ we have 
 $a_{iP} = \theta_{iP,\eta}^2 a_{iU} $ for some $\theta_{iP, U} \in F_{0U, P}^*$.
 Hence, there exist $\theta_{iP} \in F_{0P}^*$ and $ \theta_{iU} \in F_{0U}^*$ such that
 $\theta_{iP,\eta} = \theta_{iP} \theta_{iU}$ (cf. \cite[The proof of Proposition 7.4]{PPS}). Thus  $a_{iP}\theta_{iP}^{-2}  = a_{iU}\theta_{iU}^2$ for all branches $(U, P)$.
 Hence there exist $a_i \in F_0^*$ such that $a_i = a_{iP} \in F_{0P}^*/F_{0P}^{*2}$ 
 and $a_i = a_{iU} \in  F_{0U}^*/F_{0U}^{*2}$ .  Let $L_i = F_0[X]/(X^2 - a_i)$.
 Then, by (\cite[Theorem 5.1]{h2009}), ind$(D\otimes L_i) \leq 2$ for all $i$. 
 
Let $P \in X_0$ be a closed. Since $\kappa(P)$ is a finite field, there exists 
$t_P \geq 2$ such that $\kappa(P)$ has no   $2^{t_p}$th primitive root of unity.  
Let $t > 2t_P $ for all $P \in \PP$.

Let $P \in \PP$. 
 We have $\mu_{iP} \in F_{0P}[X]/(X^2 - a_{i})$ and 
  $\mu_{iU} \in F_{0U}[X]/(X^2 - a_{i})$ such that 
  $\mu_{iP}\mu_{iU}^{-1} \in (F_{0U, P}[X]/(X^2 - a_i))^{*2^m}$ for all $m \geq 1$. 
 Hence    $\mu_{iP} = \mu_{iU} \beta_{iU, P}^{2^{2t}}$ for some $\beta_{iU, P} \in L_i \otimes F_{0U, P}$. There exist  $\beta_{iP}  \in L_i \otimes F_{0P}$ and 
  $\beta_{iU} \in L_i \otimes F_{0U}$ such that 
  $\beta_{iU, P}  = \beta_{iU} \beta_{iP}$ (cf. \cite[The proof of Proposition 7.4]{PPS}).
  In particular we have $\mu_{iP} \beta_{iP} ^{-2^{2t}} = \mu_{iU}\beta_{iU}^{2^{2t}}$
   for all branches $(U, P)$. Hence, by (\cite[Proposition 6.3]{h2010}), there exist $\mu_i \in L_i$ such that 
  $\mu_i = \mu_{iP} \beta_{iP} ^{-2^{2t}} = \mu_{iU}\beta_{iU}^{2^{2t}}$.
  
  Let $\lambda_1= \lambda N_{L_1/F_0}(\mu_1)^{-1}N_{L_2/F_0}(\mu_2)^{-1}$. 
 For $\zeta \in \PP \cup \UU$,  we have 
  $$
  \begin{array}{rcl}
  \lambda_1 & = &\lambda N_{L_1/F_0}(\mu_1)^{-1}N_{L_2/F_0}(\mu_2)^{-1} \\
   &  =  &  N_{L_{1\zeta}/F_{0\zeta}}(\mu_{1\zeta})
  N_{L_{2\zeta}/F_{0\zeta}}(\mu_{2\zeta})N_{L_{3\zeta}/F_{0\zeta}}(\mu_{3\zeta})N_{L_{1}/F_{0}}(\mu_{1})^{-1} N_{L_{1}/F_{0}}(\mu_{2})^{-1} \cr
  & = &  N_{L_{1\zeta}/F_{0\zeta}}(\mu_{1\zeta} \mu_1^{-1})
  N_{L_{2\zeta}/F_{0\zeta}}(\mu_{2\zeta} \mu_2^{-1})N_{L_{3\zeta}/F_{0\zeta}}(\mu_{3\zeta}) 
  \end{array}
  $$
 Since $ N_{L_{1\zeta}/F_{0\zeta}}(\mu_{1\zeta} \mu_1^{-1})
  N_{L_{2\zeta}/F_{0\zeta}}(\mu_{2\zeta} \mu_2^{-1}) = x_\zeta^{2^{2t}}$ for some $ x_\zeta \in 
  F_{0\zeta}$, 
  we have 
  $\lambda_1 = N_{L_{3\zeta}/F_{0\zeta}}( x_\zeta^{2^{2t-1}} \mu_{3\zeta})$. 
  Since ind$(D\otimes L_{3\zeta}) \leq 2$ and 
   $\mu_{3\zeta}$ is a reduced norm from $D \otimes L_{3\zeta}$, 
   $x_\zeta^{2^{2t-1}} \mu_{3\zeta}$ is a reduced norm from $D\otimes L_{3\zeta}$. 
   Further, for every branch $(U, P)$, we have 
   $x_P^{2^{2t-1}} \mu_{3P}x_U^{-2^{2t-1}} \mu_{3U}^{-1} \in (F_{0U, P}[X]/(X^2 - a_i))^{*2^{2t-1}}$. 
  
  Replacing $\mu_{3\zeta}$ by $x_\zeta^2 \mu_{3\zeta}$ we assume that 
  $N_{L_{3\zeta}}(\mu_{3\zeta}) = \lambda_1$, 
 $  \mu_{3P}  \mu_{3U}^{-1} \in (F_{0U, P}[X]/(X^2 - a_3))^{*2^{2t-1}}$ 
   and $\mu_{3\zeta}$ is a reduced norm from 
  $D\otimes L_{3\zeta}$ for all $\zeta \in \PP \cup \UU$. 
  
Hence, as in (\cite[Proposition 6.3]{h2010}, \cite[Theorem 3.2.3]{h2014}), there exists $\mu_3 \in L_3 = F_{0}[X]/(X^2 - a_3)$
 such that $N_{L_3/F_0}(\mu_3) = \lambda_1$ and $\mu_3$ is a reduced norm from 
 $D\otimes L_3$.
 
 Therefore $a_i$ and $\mu_i$ have the required properties. 
\end{proof}
\begin{theorem} \label{thm9.2}
        Let $F_0$ be the function field of a $p$-adic curve. Let $F$ be a quadratic field extension of $F_0$. Let $A$ be a central simple $F$-algebra of the exponent $2$. If $A$ has a $F/F_0$-involution $\tau$ and $p \neq 2$, then $SK_1U(\tau, A)$ is trivial. 
\end{theorem}
\begin{proof}
    By (\cite[Lemma 2]{1974simple}), it can be reduced to the case that $A$ is a central division algebra over $F$. Choose an element $a\in \Sigma_{\tau}'(A^*)$ arbitrarily and write $\lambda=Nrd_{A/F}(a)$. Then $\lambda \in F_0^* \cap Nrd(D)^*$. 
    
    By Theorem \ref{thm9.1}, there are extensions $L_i$ of $F$ satisfying $ind(A \otimes_{F_0} L_i)\leq   2$ for $i=1, 2, 3$. Let $\widetilde{L_i}=L_i \otimes_{F_0} F$ and $\widetilde{A_i}=A \otimes_{F_0} L_i=A \otimes_F \widetilde{L_i}$ for  $i=1, 2, 3$. Considering the elements $\mu_1$, $\mu_2$, and $\mu_3$ founded in Theorem \ref{thm9.1}, let $Nrd_{\widetilde{A_i} /\widetilde{L_i}}(\widetilde{d_i})=\mu_i$ for $i=1, 2, 3$ and some $\widetilde{d_i} \in \widetilde{A_i}^*$. 
    
    Since $SUK_1(\widetilde{A_i}, \tau \otimes id)$ is trivial (cf. \cite[Proposition 17.27]{knus})     and $\mu_i \in L_i$ for $i=1, 2, 3$, we have $\widetilde{d_i} \in \Sigma_{\tau \otimes id}(\widetilde{A_i}^*)$  . By the norm principle (\cite[Proposition 4.3]{ER}), $N_{L_i/F_0}(\mu_i)=Nrd_{A/F}(d_i)$  for some $d_i \in \Sigma_{\tau}(A^*)$ where $i=1, 2,$or $ 3$ . Therefore, by Theorem \ref{thm9.1}, $\lambda= \prod_i N_{L_i/F_0}(\mu_i) = Nrd_{A/F}(\prod_i d_i)=Nrd_{A/F}(a)$.

    Since $ind(D)\leq 4$ and $cd(F) \leq 3$, $SK_1(A)$ is trivial (cf. \cite[Theorem 16.6]{pierce},  \cite{Wang}, \cite[Chapter 17]{knus}, \cite{Merk} ). Then $a^{-1} \cdot \prod_i d_i \in SL_1(A)=[A^*, A^*] \subset \Sigma_{\tau}(A^*)$ (cf. \cite[Proposition 17.26]{knus}). Since $\prod_i d_i \in \Sigma_{\tau}(A^*)$, $a \in \Sigma_{\tau}(A^*)$.
\end{proof}

\vspace{2cm}


\end{document}